\documentclass{amsart}
\usepackage{latexsym,amscd,amssymb,url}
\usepackage{extarrows}
\usepackage{verbatim}
\usepackage{dsfont}
\usepackage{tikz-cd}
\usepackage[all,cmtip]{xy}  
\usepackage{graphicx}
\usepackage{xcolor}
\usepackage{enumitem} 
\definecolor{dblue}{rgb}{0,0,.6}
\usepackage[colorlinks=true, linkcolor=dblue, citecolor=dblue, filecolor = dblue, menucolor = dblue, urlcolor = dblue]{hyperref}

\numberwithin{equation}{section}

\newtheorem{theorem}{Theorem}[section]

\theoremstyle{plain}

\newtheorem{claim}{Claim}

\newtheorem{corollary}[theorem]{Corollary}

\newtheorem{example}[theorem]{Example}

\newtheorem{lemma}[theorem]{Lemma}

\newtheorem{proposition}[theorem]{Proposition}
\newtheorem{remark}[theorem]{Remark}

\newcommand{\del}{\partial}
\newcommand{\Z}{\mathbb Z}
\newcommand{\Q}{\mathbb Q}
\newcommand{\A}{\mathbb A}

\newcommand{\C}{\mathbb C}

\newcommand{\im}{\operatorname{im}}

\newcommand{\Pic}{\operatorname{Pic}}

\newcommand{\Fix}{\operatorname{Fix}}

\newcommand{\id}{\operatorname{id}}

\newcommand{\Spec}{\operatorname{Spec}}

\newcommand{\pr}{\operatorname{pr}}
\newcommand{\NS}{\operatorname{NS}}

\newcommand{\codim}{\operatorname{codim}}

\newcommand{\rat}{\operatorname{rat}}

\newcommand{\Br}{\operatorname{Br}}

\newcommand{\CH}{\operatorname{CH}}

\newcommand{\cl}{\operatorname{cl}}

  \newcommand{\coker}{\operatorname{coker}}
  \newcommand{\Griff}{\operatorname{Griff}}  
     
\newcommand{\tors}{\operatorname{tors}}

\newcommand{\alg}{\operatorname{alg}}

\newcommand{\et}{\text{\'et}}

\newcommand{\dashedlongrightarrow}{\xymatrix@1@=15pt{\ar@{-->}[r]&}}
\renewcommand{\longrightarrow}{\xymatrix@1@=15pt{\ar[r]&}}
\renewcommand{\mapsto}{\xymatrix@1@=15pt{\ar@{|->}[r]&}}
\renewcommand{\twoheadrightarrow}{\xymatrix@1@=15pt{\ar@{->>}[r]&}}
\newcommand{\hooklongrightarrow}{\xymatrix@1@=15pt{\ar@{^(->}[r]&}}
\newcommand{\congpf}{\xymatrix@1@=15pt{\ar[r]^-\sim&}}
\renewcommand{\cong}{\simeq}

\begin{document}    
\title[Torsion in Griffiths Groups]{Torsion in Griffiths Groups}
\author{Theodosis Alexandrou} 
\address{Institut für Mathematik, Humboldt-Universität zu Berlin, Rudower Chaussee 25, 10099 Berlin, Germany}
\email{theodosis.alexandrou@hu-berlin.de} 
\date{\today} 
 \begin{abstract} We show that for any integer $n\geq2$ there is a smooth complex projective variety $X$ of dimension $5$ whose third Griffiths group $\Griff^{3}(X)$ contains infinitely many torsion elements of order $n$. This generalises a recent theorem of Schreieder who proved the result for $n=2$.\end{abstract}
\maketitle 
\section{Introduction}
 The Griffiths group $\Griff^i(X)$ of a smooth complex projective variety $X$ is the group of nullhomologous codimension $i$ cycles modulo algebraic equivalence. For divisors and zero-cycles these two equivalence relations of algebraic cycles coincide and so $\Griff^i(X)$ is interesting only if $2\leq i \leq \dim X-1$.\par The terminology of this group is due to the famous example of Griffiths \cite{griffiths}, who showed that $\Griff^2(X)$ is not a torsion group for a very general quintic $X\subset\mathbb{P}^4_{\mathbb{C}}$. Clemens \cite{clemens} proved later that in this case $\Griff^2(X)$ has in fact infinite rank. Further examples of infinite dimensionality of $\Griff^i(X)\otimes \mathbb{Q}$ were constructed by Nori \cite{nori} and Voisin \cite{voisin-calabi-yau}.\par Totaro \cite{totaro-annals} showed that the vector space $\Griff^{i}(X)/\ell$ can be infinite dimensional for all prime numbers $\ell$ and all $2\leq i\leq \dim X-1$. This builds on earlier works of Schoen \cite{schoen-modn} and Rosenschon-Srinivas \cite{RS}, where there are constraints on the prime $\ell$. The result of Totaro also relies on Griffiths' method \cite{griffiths}, a theorem of Bloch-Esnault \cite{Bloch-Esnault} and Nori's approach \cite{nori} of the proof of Clemens' theorem.
 
 \par Schoen \cite{schoen-torsion} proved that the torsion subgroup of the Griffiths groups is generally not zero. Merkurjev--Suslin \cite{MS} used Bloch's map \cite{bloch-compositio} to show that the $n$-torsion in $\Griff^{2}(X)$ is always finite. Thus, it was a natural problem in the field whether $\Griff^{i}(X)$ for $i\geq3$ can have infinite $n$-torsion, cf. \cite{schoen-torsion}. This was recently solved by Schreieder \cite{Sch-griffiths}, who produced examples with infinite 2-torsion in their third Griffiths group. The question of infinite $\ell$-torsion for odd primes $\ell$ was left open. Building on Schreieder’s approach, in this paper, we solve the problem in the following strong sense:
 
 \begin{theorem}\label{thm:ell-torsion-griffiths} Let $n\geq2$ be an integer. There is a smooth complex projective variety $X$ of dimension 5 whose third Griffiths group $\Griff^{3}(X)$ contains infinitely many torsion elements of order $n$. Moreover, these elements are linearly independent modulo $n$, i.e. in $\Griff^{3}(X)/n$.\end{theorem}
 \par The above result implies that for any finite set \textbf{P} of prime numbers, there is a smooth complex projective $5$-fold $X$, such that the vector space $\Griff^{3}(X)[\ell^{\infty}]/\ell$ is infinite dimensional for all $\ell\in\textbf{P}$. 
 
 \par Taking products with projective spaces, we obtain the following generalization in higher dimensions. 
 \begin{corollary}\label{cor:ell-torsion-griffiths} For any integers $m\geq 5$ and $n\geq 2$, there is a smooth complex projective $m$-fold $X$ whose $j$-th Griffiths group $\Griff^{j}(X)$ contains infinitely many torsion elements of order $n$ for all $3\leq j\leq m-2$. Moreover, these elements are linearly independent in $\Griff^{j}(X)/n$.\end{corollary} 
 
 \par The examples in Theorem \ref{thm:ell-torsion-griffiths} are given by a product $X:=S\times JC$, where $S$ is a $\Z/n$-quotient of a carefully chosen smooth complete intersection of multidegree $(n,n,n)$ in $\mathbb{P}^{5}_{\C}$ (see $\S$\ref{sec:cyclic-quotient}) and $C$ is a very general genus $3$ curve. The surface $S$ is somewhat special as such admits an exceptional degeneration, see $\S \ref{sb:degenerations}$ below. For example,if $n=2$, then $S$ is an Enriques surface (as used in \cite{Sch-griffiths}) that degenerates into a flower pot. For $n>2$, the surface $S$ is of general type; its geometric genus is positive and grows polynomially in $n$, see Proposition \ref{prp:invariants-X}.
 
 \par Let $Y$ be a smooth complex projective variety and let $\NS(Y):=\Pic(Y)/\Pic^{0}(Y)$ denote its N\'eron Severi group. The following injectivity result is crucial for us.
 
 \begin{theorem}\label{thm:injectivity-result} Let $n\geq2$ be an integer. Let $Z$ be a complex smooth projective variety. Then there is a complex smooth projective surface $S$ with $\NS(S)_{\tors}\cong\Z/n$, such that the exterior product map \begin{align}\label{ext-prod-griffiths}\Griff^{j}(Z)/n\longrightarrow\frac{\Griff^{j+1}(S\times Z)[n]}{n\Griff^{j+1}(S\times Z)[n^{2}]},\ [z]\mapsto [\mathcal{L}\times z]\end{align} is injective, where $0\neq\mathcal{L}\in\NS(S)_{\tors}$ is any generator.\end{theorem}
 
Theorem \ref{thm:ell-torsion-griffiths} follows from Theorem \ref{thm:injectivity-result} because the Jacobian of a very general genus $3$ curve $C$ has the property that $\Griff^{2}(JC)/n$ has infinitely many elements of order $n$ by a recent result of Totaro \cite{totaro-annals} and the Merkurjev--Suslin theorem \cite{MS} (see Theorem \ref{thm:infinite-modulo-n}).

\subsection{Special degenerations of surfaces}\label{sb:degenerations} Apart from a general injectivity result that works for all primes, the second main ingredient in \cite{Sch-griffiths} was the construction of degenerations of Enriques surfaces such that the Brauer class of the geometric generic fibre extends throughout the family, while the special fibre splits up into several components that are all rational and hence the Brauer class restricts to zero on the components of the special fibre. These degenerations of Enriques surfaces are known in the literature as flower-pot, see \cite[Proposition 3.3.1(3)]{persson}. They are exceptional as such do not admit (up to birational equivalence) an \'etale 2:1 cover by a Kulikov degeneration (see \cite{kulikov}) of $K3$ surfaces.

\par The main task within this paper is thus, to generalize the above construction to other surfaces. Namely, to prove Theorem \ref{thm:injectivity-result} we need to exhibit for every $n\geq2$, a surface $S$ with $\NS(S)_{\tors}\cong\Z/n$ that admits a similar degeneration, i.e. there is a Brauer class $\alpha\in\Br(S)[n]$, whose image via the Bockstein map generates the torsion in $\bigoplus_{\ell}H^{3}(S,\Z_{\ell}(1))$ and extends across the whole family, while its restriction to the components of the special fibre is zero. This is achieved by the following result.

\begin{theorem}\label{thm:extend-Brauer-class} Let $\kappa$ be an algebraically closed field of characteristic zero and let $n\geq2$ be an integer. There is a regular flat projective scheme $\mathcal{S}\to\Spec\kappa[[t]]$ such that:\begin{enumerate}
    \item\label{itm:generic-fibre} the geometric generic fibre $S_{\bar{\eta}}$ is a smooth projective surface with $\NS(S_{\bar{\eta}})_{\tors}\cong\Z/n$.
    \item\label{itm:special-fibre} the special fibre $S_{0}=\sum_{i=1}^{r}S_{0i}$ is a reduced simple normal crossing divisor whose dual graph $\Gamma$ is a chain, i.e. the double intersection $S_{0i}\cap S_{0i+1}$ is a smooth and irreducible curve while all intersections $S_{0i}\cap S_{0j}$ with $j\notin\{i-1,i,i+1\}$ are empty.
    \item\label{it:extend-brauer-class} there is a class $\alpha\in\Br(\mathcal{S})[n]$, such that\begin{enumerate}
        \item \label{itm:extend-class}$\delta(\alpha|_{S_{\bar{\eta}}})\in \bigoplus_{\ell}H^{3}(S_{\bar{\eta}},\Z_{\ell}(1))_{\tors}\cong\Z/n$ is a generator, where $\delta$ is the Bockstein map defined on the Brauer group $\Br(S_{\bar{\eta}})$ and such that \item\label{itm:rest-zero} for any component $S_{0i}$ of the special fibre, $\alpha|_{S_{0i}}=0\in\Br(S_{0i})$.\end{enumerate}\end{enumerate}\end{theorem} 

\par The existence of the above degenerations seems to be slightly counter intuitive as one typically expects that Brauer classes ramify when the varieties break up into components. Their construction is inspired by the one of flower pot of Enriques that was given by Persson, see \cite[Appendix 2]{persson}.\par It would be interesting to know if Theorem \ref{thm:extend-Brauer-class} holds for some degenerations of some well-known surfaces e.g. Godeaux surface. For instance, the property \eqref{itm:extend-class} completely fails (see \cite[Lemma 5.3]{theo}) for the degenerations of Godeaux surfaces constructed in \cite[Appendix 3]{persson}.

\subsection{Outline of the main arguments} We briefly explain some of the main steps of the proof of Theorem \ref{thm:extend-Brauer-class}. Let $Y\subset\mathbb{P}^{5}_{k}$ be a smooth complete intersection of multidegree $(n,n,n)$ that is invariant and fixed point free by the automorphism $\varphi_{n}$, see \eqref{eq:varphi}. Then the crucial step is to degenerate $Y$ to some special singular complete intersection $Y_{0}\subset\mathbb{P}^{5}_{\kappa}$ (see Lemma \ref{lem:resolution-Y'}) that is invariant by $\varphi_{n}$. The singular locus of $Y_{0}$ consists of a single point $p_{0}$ of multiplicity $n$ and the automorphism $\varphi_{n}$ fixes exactly this point in $Y_{0}$. The total space $\mathcal{Y}$ of this degeneration can be chosen to be regular and so blowing up $p_{0}$ in our family we obtain $\mathbb{P}^{2}_{\kappa}$ with multiplicity $n$. The automorphism $\varphi_{n}$ lifts naturally on the blow-up and we will see that in fact acts trivially on the exceptional divisor. Quotienting out yields the desired degeneration.  
        
\par For $n>2$, the property \eqref{itm:extend-class} is more subtle. In fact, the surface $S:=S_{\bar{\eta}}$ has positive geometric genus and thus a class $\alpha\in\Br(S)[n]$ with the property that $\delta(\alpha)$ generates the torsion in $\bigoplus_{\ell}H^{3}(S,\Z_{\ell}(1))$ is not canonically given any longer. Instead, the functoriality of the Bockstein map (see \eqref{eq:bockstein-sequence}) with respect to pull-backs reduces our task to showing $$H^{3}(\mathcal{S},\Z_{\ell}(1))_{\tors}\to H^{3}(S,\Z_{\ell}(1))_{\tors}$$ is surjective for all primes $\ell$ dividing $n$. One of the advantages of this reduction step is that we can apply the proper base change theorem to compute $H^{3}(\mathcal{S},\Z_{\ell}(1))$ (see Claim \ref{claim-3}) and in particular its torsion subgroup. The existence of a lift $\beta\in H^{3}(\mathcal{S},\Z_{\ell}(1))$ of a generator of the cyclic group $H^{3}(S,\Z_{\ell}(1))_{\tors}$ follows from the Gysin sequence. The intricate part of the proof is then to show that the class $\beta$ can be chosen to be torsion, so that it comes from a Brauer class on $\mathcal{S}$.\par The property \eqref{itm:rest-zero} of Theorem \ref{thm:extend-Brauer-class} is more involved than the case of flower pot of Enriques, i.e. $n=2$. This is because the special fibre $S_{0}$ contains a single component (say $W$), whose Brauer group is infinite if $n>2$ (see Proposition \ref{prop:geometric-genus-W}). We will be able to take care of this matter by adding some integral cohomology class to a lift $\tilde{\alpha}\in H^{2}(\mathcal{S},\mu_{n})$ of the Brauer class on $\mathcal{S}$ from \eqref{itm:extend-class}, i.e., $\tilde{\alpha}|_{W}=0$ up to adding a cohomology class from $H^{2}(\mathcal{S},\Z_{n}(1)):=\prod_{\ell|n}H^{2}(\mathcal{S},\Z_{\ell}(1))$. Note that this modification on the class $\alpha$ has no effect on its value via the Bockstein map (cf. \eqref{eq:bockstein-sequence}).

\section{Preliminaries}

\subsection{Notations}The $n$-torsion subgroup of an abelian group $G$ is denoted by $G[n]$. The subgroup $G[n^{\infty}]$ consists of the elements that are annihilated by some power of $n$. We also denote by $G_{div}$ the maximal divisible subgroup of $G$. If $\psi\colon H\to G$ is a homomorphism of abelian groups, then we denote by slight abuse of notation $G/H:=\coker(\psi)$.\par Let $k$ be a field. An algebraic scheme $X$ over $k$ is a separated scheme of finite type over $k$ and it is equi-dimensional if its irreducible components have all the same dimension. A variety over $k$ is an integral algebraic $k$-scheme. For an equi-dimensional algebraic $k$-scheme $X$ and any $0\leq i\leq \dim X$, we let $X^{(i)}:=\{x\in X|\codim_{X}(x):=\dim X-\dim\overline{\{x\}}=i\}$ be the set of all codimension $i$ points in $X$.\subsection{Borel-Moore cohomology}Fix a field $k$ and a prime $\ell$ invertible in $k$. For an algebraic $k$-scheme $X$ of dimension $d_{X}$ and $A\in\{\Z/\ell^{r},\Z_{\ell},\Q_{\ell},\Q_{\ell}/\Z_{\ell}\}$, we let $$H^{i}(X,A(n)):=H^{i}_{BM}(X,A(n))$$ be the twisted Borel-Moore pro-\'etale cohomology, as defined in \cite[(6.13)-(6.15)]{Sch-refined} (see \cite[Proposition 6.6]{Sch-refined}). The main properties of this functor are listed in \cite[\S4]{Sch-refined}.\par If $X$ is smooth and equi-dimensional, then by \cite[Lemma 6.5]{Sch-refined}, there are canonical isomorphisms \begin{align}\label{eq:borel-moore} H^{i}(X,\Z/\ell^{r}(n))\cong H^{i}(X_{\text{\'et}},\mu_{\ell^{r}}^{\otimes n})\ \text{and}\ H^{i}(X,\Z_{\ell}(n))\cong H^{i}_{cont}(X_{\text{\'et}},\Z_{\ell}(n)), \end{align} where $H^{i}_{cont}$ denotes Jannsen's continuous \'etale cohomology (see \cite{jannsen}).\par More generally, if $X$ admits a closed embedding $\iota\colon X\hookrightarrow Y$ into a smooth and equi-dimensional algebraic $k$-scheme $Y$ of dimension $d_{Y}$, then there is a canonical identification with continuous \'etale cohomology with support \begin{align}\label{eq:borel-moore-support} H^{i}(X,A(n))\cong H^{i+2c}_{X,cont}(Y,A(n+c)),\end{align} where $c=d_{Y}-d_{X}$ (see \cite[Remark 2.1]{th-sch} and \cite[Lemma A.1]{Sch-moving}).\subsection{Refined unramified cohomology}\label{sb:refined}For an algebraic $k$-scheme $X$, we have an increasing filtration: $$F_{0}X\subset F_{1}X\subset F_{2}X\subset\dots\subset F_{\dim X}X=X,$$ where $F_{j}X$ is given by the inverse limit of all open subsets $U\subset X$, with $\dim(X\setminus U)<\dim X-j$. Recall that for an open subset $U\subset X$, with $\dim U=\dim X$, there are restriction maps $H^{*}(X,A(n))\to H^{*}(U,A(n))$ and one defines $$H^{*}(F_{j}X,A(n)):=\lim_{\substack{\longrightarrow\\ F_{j}X\subset U\subset X}}H^{*}(U,A(n)).$$Here the direct limit is taken over the open subsets $U\subset X$ that make up $F_{j}X$ (see \cite[\S 5]{Sch-refined}).\par If $m\geq j$, then we have restriction maps $H^{*}(F_{m}X,A(n))\to H^{*}(F_{j}X,A(n))$ and we set $$F^{m}H^{*}(F_{j}X,A(n)):=\im(H^{*}(F_{m}X,A(n))\to H^{*}(F_{j}X,A(n))).$$ \par We define the $j$-th refined unramified cohomology of $X$ with values in $A(n)$ by $$H^{*}_{j,nr}(X,A(n)):=F^{j+1}H^{*}(F_{j}X,A(n)).$$ \par For a scheme point $x\in X$, we write $H^i(x,A(n)):=H^i(F_0\overline{\{ x\}},A(n))$, where $\overline{\{ x\}}\subset X$ denotes the closure of $x$. Note that  $H^0(x,A(0))=A\cdot [x]$, where $[x]\in H^0(x,A(0))$ denotes the fundamental class of $x$, cf.\ \cite[(P3) in Definition 4.2 and Proposition 6.6]{Sch-refined}.

The Gysin sequence induces the following long exact sequence (see \cite[Lemma 5.8]{Sch-refined})
\begin{align} \label{eq:les}
 \longrightarrow
   H^i(F_jX,A(n))\longrightarrow   H^i(F_{j-1}X,A(n))  \stackrel{\del}\longrightarrow \bigoplus_{x\in X^{(j)}}H^{i+1-2j}(x,A(n-j)) \stackrel{\iota_\ast} \longrightarrow    H^{i+1}(F_jX,A(n)).
\end{align} 
We have $H^i(X,A(n))=H^i(F_jX,A(n))$ for all $j\geq \lceil i/2\rceil$ by \cite[Corollary 5.10]{Sch-refined}. Moreover, if $X$ is smooth and equi-dimensional $H^{*}_{0,nr}(X,A(n))$ coincides with the classical \cite{Sch-survey} unramified cohomology $H^{*}_{nr}(X,A(n))$.\subsection{Algebraic cycle groups}\label{sb:cycle-groups} Let $k$ be a field and let $\ell$ be a prime number invertible in $k$. For an algebraic $k$-scheme $X$ of dimension $d_{X}$, we denote by $\CH^{i}(X):=\CH_{d_{X}-i}(X)$ the Chow group of cycles of dimension $d_{X}-i$ on $X$. We set $\CH^i(X)_{\Z_\ell}:=\CH^i(X)\otimes_{\Z} \Z_\ell$. 
There is a cycle class map (see \cite[(7.1)]{Sch-refined})
$$
\cl^i_X:\CH^i(X)_{\Z_\ell}\longrightarrow H^{2i}(X,\Z_\ell(i)),
$$  
that is induced by the following pushforward map that appears in the above long exact sequence (\ref{eq:les}):
$$
 \bigoplus_{x\in X^{(i)}}H^{0}(x,\Z_\ell(0))= \bigoplus_{x\in X^{(i)}} \Z_\ell[x] \stackrel{\iota_\ast} \longrightarrow    H^{2i}(F_iX,\Z_\ell(i))=H^{2i}(X,\Z_\ell(i)).
$$
If $X$ is smooth and equi-dimensional, then the above cycle class map agrees with Jannsen's cycle class map in continuous \'etale cohomology from \cite{jannsen}, see \cite[Lemma 9.1]{Sch-refined}.

There is a natural coniveau filtration $N^\ast$ on $\CH^i(X)_{\Z_\ell}$, given by the condition that a cycle $[z]$ lies in $N^j$ if and only if  
there is a closed subset $Z\subset X$ of codimension $j$ such that $z$ is rationally equivalent to a homologously trivial cycle on $Z$, i.e.\  
$$
[z]\in \im(\ker(\cl_Z^{i-j})\longrightarrow \CH^i(X)_{\Z_\ell} ) ,
$$
see \cite[Definition 7.3]{Sch-refined}.
In view of \cite[Definition 7.2 and Lemma 7.4]{Sch-refined}, we define
$$
A^i(X)_{\Z_\ell}:=\CH^i(X)_{\Z_\ell}/N^{i-1}\CH^i(X)_{\Z_\ell} \ \ \text{and}\ \ A^i(X)[\ell^\infty]:=A^i(X)_{\Z_\ell}[\ell^\infty].
$$
If the field $k$ is algebraically closed, then $N^{i-1}\CH^i(X)_{\Z_\ell}=\CH^{i}(X)_{\alg}\otimes_{\Z}\Z_{\ell}$ by \cite[Lemma 7.5]{Sch-refined} and therefore in this case $A^i(X)_{\Z_\ell}=(\CH^i(X)/\sim_{\alg})\otimes_{\Z}{\Z_\ell}$.
\subsection{Brauer groups}\label{sb:brauer} For a scheme $X$, we let $\Br(X):=H^{2}(X_{\text{\'et}},\mathbb{G}_{m})$ be the Brauer group of $X$. If $n$ is a positive integer that is invertible on $X$, then the Kummer sequence yields a short exact sequence \begin{align}\label{eq:kummer} 0\longrightarrow\Pic(X)/n\overset{\cl^{1}}{\longrightarrow} H^{2}(X_{\text{\'et}},\mu_{n})\longrightarrow\Br(X)[n]\longrightarrow 0. \end{align} We also recall that $\Br(X)[n]=H^{2}_{nr}(X_{\text{\'et}},\mu_{n})$ for a regular variety $X$ (cf. \cite[Theorem 3.7.3]{CS-brauer}).\ \par Let $X$ be an algebraic scheme over a field $k$ and fix a prime number $\ell$ invertible in $k$. There is a long exact Bockstein sequence\begin{align}\label{eq:bockstein-sequence}
\dots\longrightarrow H^{i}(X,\Z_{\ell}(n))\overset{\times \ell^{r}}{\longrightarrow}H^{i}(X,\Z_{\ell}(n))\longrightarrow H^{i}(X,\Z/\ell^{r}(n))\overset{\delta}{\longrightarrow}H^{i+1}(X,\Z_{\ell}(n))\overset{\times \ell^{r}}{\longrightarrow}\dots\end{align} where $H^{i}(X,\Z_{\ell}(n))\longrightarrow H^{i}(X,\Z/\ell^{r}(n))$ is given by functoriality in the coefficients and where $\delta$ is called the Bockstein map (see \cite[(P5) in Definition 4.4]{Sch-refined}).
\begin{proposition}\label{prop:short-exact-sequence}Let $X$ be an algebraic scheme over a field $k$ and let $\ell$ be a prime invertible in $k$. Then the Bockstein map $\delta\colon H^{2}(X,\Q_{\ell}/\Z_{\ell}(1))\longrightarrow H^{3}(X,\Z_{\ell}(1))_{\tors}$ descends to a well-defined map $\delta\colon H^{2}_{0,nr}(X,\Q_{\ell}/\Z_{\ell}(1))\longrightarrow H^{3}(X,\Z_{\ell}(1))_{\tors}$ that fits into a short exact sequence \begin{align}\label{eq:short-exact-sequence}0\longrightarrow \frac{H^{2}_{0,nr}(X,\Q_{\ell}(1))}{H^{2}_{0,nr}(X,\Z_{\ell}(1))}\longrightarrow H^{2}_{0,nr}(X,\Q_{\ell}/\Z_{\ell}(1))\overset{\delta}{\longrightarrow} H^{3}(X,\Z_{\ell}(1))_{\tors}\longrightarrow 0.\end{align} \end{proposition}\begin{proof} Pick a class $\alpha\in H^{2}(X,\Z/\ell^{r}(1))$ whose image in $H^{2}_{0,nr}(X,\Z/\ell^{r}(1))$ is zero. It follows from the Gysin sequence \eqref{eq:les} that $\alpha=\iota_{*}\xi$ for some $\xi\in\bigoplus_{x\in X^{(1)}}\Z/\ell^{r}\cdot[x]$. By functoriality of the Bockstein sequence \eqref{eq:bockstein-sequence} with respect to proper pushforwards (see \cite[(P5) in Definition 4.4]{Sch-refined}) we obtain $\delta(\iota_{*}\xi)=\iota_{*}\delta(\xi)$. Note that the element $\delta(\xi)\in \bigoplus_{x\in X^{(1)}}H^{1}(x,\Z_{\ell}(0))$ is torsion. The group $H^{1}(x,\Z_{\ell}(0))$ being torsion free (see \cite[Lemma 5.13]{Sch-refined}) thus yields $\delta(\xi)=0$. In particular, the Bockstein map induces a well-defined map $\delta\colon H^{2}_{0,nr}(X,\Z/\ell^{r}(1))\longrightarrow H^{3}(X,\Z_{\ell}(1))[\ell^{r}]$ and taking the direct limit over $r$ one gets $\delta\colon H^{2}_{0,nr}(X,\Q_{\ell}/\Z_{\ell}(1))\longrightarrow H^{3}(X,\Z_{\ell}(1))_{\tors}$.\par Consider the following commutative diagram with exact rows \begin{center}\begin{tikzcd}
0 \arrow[r] & {H^{2}(X,\Z_{\ell}(1))/\ell^{r}} \arrow[r] \arrow[d, "\pi_{1}", two heads] & {H^{2}(X,\Z/\ell^{r}(1))} \arrow[r, "\delta"] \arrow[d, "\pi_{2}", two heads] & {H^{3}(X,\Z_{\ell}(1))[\ell^{r}]} \arrow[r] \arrow[d, "\pi_{3}", two heads]   & 0 \\
 & {H^{2}_{0,nr}(X,\Z_{\ell}(1))/\ell^{r}} \arrow[r]                          & {H^{2}_{0,nr}(X,\Z/\ell^{r}(1))} \arrow[r]                                      & {\frac{H^{2}_{0,nr}(X,\Z/\ell^{r}(1))}{H^{2}_{0,nr}(X,\Z_{\ell}(1))}} \arrow[r] & 0,
\end{tikzcd}\end{center} where the upper row corresponds to the short exact sequence induced by \eqref{eq:bockstein-sequence} and where the $\pi_{i}$ are surjective simply because $H^{2}(X,A(n))=H^{2}(F_{1}X,A(n))\twoheadrightarrow H^{2}_{0,nr}(X,A(n))$ is onto.\par We wish to show that the lower row can be completed to a short exact sequence. It suffices to prove injectivity for $H^{2}_{0,nr}(X,\Z_{\ell}(1))/\ell^{r}\longrightarrow H^{2}_{0,nr}(X,\Z/\ell^{r}(1))$. To this end, pick a class $\alpha\in H^{2}(X,\Z_{\ell}(1))$ and assume that its reduction $\bar{\alpha}$ modulo $\ell^{r}$ restricts to zero in $H^{2}_{0,nr}(X,\Z/\ell^{r}(1))$. It follows from the Gysin sequence \eqref{eq:les} that $\bar{\alpha}=\iota_{*}\bar{\xi}$ for some $\xi\in\bigoplus_{x\in X^{(1)}}\Z_{\ell}\cdot[x]$, where we denote by $\bar{\xi}$ the reduction of this cycle modulo $\ell^{r}$. Therefore $\alpha=\iota_{*}\xi\in H^{2}(X,\Z_{\ell}(1))/\ell^{r}$ and by exactness of \eqref{eq:les}, the image of the class $\alpha$ in $H^{2}_{0,nr}(X,\Z_{\ell}(1))/\ell^{r}$ is zero.\par Next we prove that $\pi_{3}$ is an isomorphism. Recall that this map is onto and so we only need to prove is injective. In the argument of the previous paragraph, we have seen that the kernel $\ker(\pi_{1})$ surjects onto $\ker(\pi_{2})$. Since $\pi_{1}$ is surjective, the result is an immediate consequence of the serpents lemma.\par Finally, \eqref{eq:short-exact-sequence} follows from the short exact sequence \begin{align}\label{eq:short-exact-sequence-2}
    0\longrightarrow
H^{2}_{0,nr}(X,\Z_{\ell}(1))/\ell^{r}\longrightarrow H^{2}_{0,nr}(X,\Z/\ell^{r}(1))\overset{\delta}{\longrightarrow} H^{3}(X,\Z_{\ell}(1))[\ell^{r}]\longrightarrow 0,\end{align} by taking the direct limit over $r$. This completes the proof.\end{proof}
If $X$ is smooth and equi-dimensional, then \eqref{eq:short-exact-sequence} yields \begin{align}\label{eq:extension-brauer} 
0\longrightarrow \Br(X)[\ell^{\infty}]_{div}\longrightarrow\Br(X)[\ell^{\infty}]\overset{\delta}{\longrightarrow} H^{3}_{cont}(X,\Z_{\ell}(1))_{\tors}\longrightarrow 0,\end{align} where $\Br(X)[\ell^{\infty}]_{div}:=\frac{H^{2}_{0,nr}(X,\Q_{\ell}(1))}{H^{2}_{0,nr}(X,\Z_{\ell}(1))}$ is a divisible group. Note that the latter group, in general, need not
be the maximal divisible subgroup of $\Br(X)[\ell^{\infty}]$. This is always the case if $k$ is separably closed, see \cite[Proposition 5.2.9]{CS-brauer}.

\subsection{Strictly semi-stable models}\label{subsec:semi-stable} Let $R$ be a discrete valuation ring with fraction field $k$ and residue field $\kappa$. A proper flat $R$-scheme $\mathcal{X}\to\Spec R$ is called \textit{strictly semi-stable} if the generic fibre $X_{\eta}:=\mathcal{X}\times_{R} k$ is smooth and the special fibre $X_{0}:=\mathcal{X}\times_{R}\kappa$ is a geometrically reduced simple normal crossing divisor on $\mathcal{X}$, i.e. the irreducible components $X_{0i}$ are all smooth Cartier divisors and for every $n$, the intersection of $n$ distinct components is either empty or smooth and equi-dimensional of codimension $n$ in $\mathcal{X}$. Note that the total space $\mathcal{X}$ is in fact regular (see \cite[Remark 1.1.1 and Remark 1.1.2]{hart}). We say that the dual graph $\Gamma$ of the special fibre $X_{0}$ is a \textit{chain} if the scheme-theoretic intersection of two consecutive components $X_{0i}\cap X_{0i+1}$ is non-empty while all intersections $X_{0i}\cap X_{0j}$ with $j\notin\{i-1,i,i+1\}$ are empty.
\subsection{An injectivity result of Schreieder} Let $Y$ be a smooth projective variety over an algebraically closed field $k$ and let $n$ be a positive integer invertible in $k$. We shall use the notation \begin{align} E^{j}_{n}(Y):=\ker(\CH^{j}(Y)/n\overset{\cl^{j}_{Y}}{\longrightarrow} H^{2j}(Y,\mu_{n}^{\otimes j})).\end{align}

\par The following result is contained in \cite{Sch-griffiths}.

\begin{theorem}{$($\cite{Sch-griffiths}$)$}\label{thm:injectivity-exterior-product-a} Let $\kappa$ be an algebraically closed field of characteristic zero and fix an algebraic closure $k$ of the fraction field $\kappa((t))$. Let $\mathcal{X}\to\Spec\kappa[[t]]$ be a proper strictly semi-stable scheme whose geometric generic fibre $X_{\bar{\eta}}$ is a surface with $\NS(X_{\bar{\eta}})_{\tors}\cong\Z/n$ for some $n\geq2$. Assume that the following properties hold \begin{itemize}
    \item [(P1)]\label{itm:P1} there is a class $\alpha\in\Br(\mathcal{X})[n]$, such that $\delta(\alpha|_{X_{\bar{\eta}}})\in \bigoplus_{\ell|n}H^{3}(X_{\bar{\eta}},\Z_{\ell}(1))_{\tors}\cong\Z/n$ is a generator, where $\delta$ is the Bockstein map defined on $\Br(X_{\bar{\eta}})$ and such that
    \item [(P2)]\label{itm:P2} for each component $X_{0i}$ of the special fibre, $\alpha|_{X_{0i}}=0\in\Br(X_{0i})$.
\end{itemize} Then for every smooth projective variety $Z$ over $\kappa$ the exterior product map \begin{align}\label{eq:exterior-product-a}
\CH^{j}(Z)/n\longrightarrow \CH^{j+1}(X_{\bar{\eta}}\times_{k}Z_{k})/n,\ [z]\mapsto [\mathcal{L}\times z_{k}]\end{align} is injective, where $\mathcal{L}\in\NS(X_{\bar{\eta}})_{\tors}$ is any generator.\end{theorem}

\begin{proof} By \cite[Remark 6.2]{Sch-griffiths} and \eqref{eq:short-exact-sequence-2}, the conditions $(\text{C}1)$ and $(\text{C}2)$ in \cite[Theorem 6.1]{Sch-griffiths} can be replaced by $(\text{P}1)$ and $(\text{P}2)$, respectively. Write $n=\ell^{r_{1}}_{1}\ell^{r_{2}}_{2}\dots\ell^{r_{s}}_{s}$ for the prime factorisation. Let $\tilde{\delta}\colon H^{1}(X_{\bar{\eta}},\mu_{\ell_{i}^{r_{i}}})\to H^{2}(X_{\bar{\eta}},\mu_{\ell_{i}^{r_{i}}})$ be the composition of the Bockstein map with reduction modulo $\ell_{i}^{r_{i}}$ and recall by \cite[Lemma 3.2]{Sch-griffiths} that $$\tilde{\delta}(H^{1}(X_{\bar{\eta}},\mu_{\ell_{i}^{r_{i}}}))\cong \NS(X_{\bar{\eta}})[\ell_{i}^{r_{i}}]/\ell_{i}^{r_{i}}\NS(X_{\bar{\eta}})[\ell_{i}^{2r_{i}}].$$ Since $\NS(X_{\bar{\eta}})_{\tors}\cong\Z/n$, we find $\tilde{\delta}(H^{1}(X_{\bar{\eta}},\mu_{\ell_{i}^{r_{i}}}))\cong \NS(X_{\bar{\eta}})[\ell_{i}^{r_{i}}]\cong\Z/\ell_{i}^{r_{i}}$. The injectivity of the composite\begin{align*} \NS(X_{\bar{\eta}})[\ell_{i}^{r_{i}}]\otimes E^{j}_{\ell_{i}^{r_{i}}}(Z)\longrightarrow \NS(X_{\bar{\eta}})[\ell_{i}^{r_{i}}]\otimes E^{j}_{\ell_{i}^{r_{i}}}(Z_{k})\longrightarrow \CH^{j+1}(X_{\bar{\eta}}\times_{k}Z_{k})/\ell_{i}^{r_{i}}\end{align*} thus follows from \cite[Theorem 6.1]{Sch-griffiths} applied to the free $\Z/\ell_{i}^{r_{i}}$-module $M=\tilde{\delta}(H^{1}(X_{\bar{\eta}},\mu_{\ell_{i}^{r_{i}}}))$ and from the commutativity of the diagram \cite[(3.3)]{Sch-griffiths}, see \cite[Lemma 3.3 and Theorem 2.7]{Sch-griffiths}. The exterior product map $$E^{j}_{n}(Z)\longrightarrow \CH^{j+1}(X_{\bar{\eta}}\times_{k}Z_{k})/n,\ [z]\mapsto [\mathcal{L}\times z_{k}]$$ in turn is injective by the Chinese remainder theorem. 
\par To conclude we need to show that the kernel of \eqref{eq:exterior-product-a} is really contained in $E^{j}_{n}(Z)$. To see this, consider the natural projections $p\colon X_{\bar{\eta}}\times_{k} Z_{k}\to X_{\bar{\eta}}$ and $q\colon X_{\bar{\eta}}\times_{k} Z_{k}\to Z_{k}$ and note that $$\cl^{j+1}_{X_{\bar{\eta}}\times_{k} Z_{k}}([\mathcal{L}\times z])=p^{*}\cl_{X_{\bar{\eta}}}^{1}([\mathcal{L}])\cup q^{*}\cl_{Z_{k}}^{j}([z])\in H^{2j+2}(X_{\bar{\eta}}\times_{k} Z_{k},\Z/n(j+1)).$$ 
Put $\gamma:=\cl_{X_{\bar{\eta}}}^{1}([\mathcal{L}])\in H^{2}(X_{\bar{\eta}},\Z/n(1))$ and $\beta:=\cl_{Z_{k}}^{j}([z])\in H^{2j}(Z_{k},\Z/n(j))$. Since $\gamma\neq 0$, we find that $p^{*}\gamma\cup q^{*}\beta=0$ if and only if $\beta=0$. Indeed, Poincar\'e duality yields a class $\bar{\gamma}\in H^{2}(X_{\bar{\eta}},\Z/n(1))$ with $\gamma\cup\bar{\gamma}=\cl^{2}_{X_{\bar{\eta}}}(pt)$ and we can recover the class $\beta$ by 
$$q_{*}(p^{*}\bar{\gamma}\cup p^{*}\gamma \cup q^{*}\beta)=q_{*}(p^{*}\cl_{X_{\bar{\eta}}}^{2}(pt)\cup q^{*}\beta)=\beta\in H^{2j}(Z_{k},\Z/n(j)).$$ 
Hence, we see that $[\mathcal{L}\times z]=0$ implies that $\cl_{Z_{k}}^{j}([z])=0\in H^{2j}(Z_{k},\Z/n(j))$. In particular, by the invariance of \'etale cohomology under separably closed field extensions, the latter relation is equivalent to $[z]\in E^{j}_{n}(Z)$. The proof is finally complete.\end{proof}

\subsection{Griffiths groups modulo $n$} We need the following variation of Totaro's result in \cite{totaro-annals}.\begin{theorem}{$($\cite{totaro-annals}$)$}\label{thm:infinite-modulo-n} There is a smooth quartic curve $C\subset\mathbb{P}^{2}_{\C}$, such that for every $n\geq2$ the group $\Griff^{2}(JC)/n$ contains infinitely many elements of order $n$.\end{theorem}

\begin{proof} Clearly by the Chinese remainder theorem it suffices to prove the theorem for all integers of the form $n=\ell^{r}$, where $\ell$ is a prime and $r\geq1$ an integer. By \cite[Theorem 3.1]{totaro-annals}, there is a smooth quartic curve $C\subset\mathbb{P}^{2}_{\C}$, such that $\CH^{2}(JC)/\ell$ is infinite for all prime numbers $\ell$. Recall that the group of algebraically trivial cycles modulo rational equivalence is divisible and so $A^{2}(JC)/\ell\cong\CH^{2}(JC)/\ell$. In particular, we find that $A^{2}(JC)/\ell$ is infinite for all primes $\ell$. The infiniteness of $A^{2}(JC)/\ell$ in this case arises from pulling back the Ceresa cycle by infinitely many distinct isogenies. These cycles are known to be homologically trivial (see \cite{ceresa}) and thus $\Griff^{2}(JC)/\ell$ is infinite dimensional for all primes $\ell$.

\par It remains to treat the case $r\geq2$. Recall by \cite[$\S$18]{MS} that the group $\Griff^{2}(JC)[\ell]$ is finite for all prime numbers $\ell$. Then the result follows from the following elementary lemma about abelian groups.
\end{proof}

\begin{lemma}\label{lem:elementary} Let $\ell$ be a prime number. Let $A$ be an abelian group such that the $\ell$-torsion subgroup $A[\ell]$ is finite and such that $A/\ell$ is infinite. Then for all $r\geq1$, the group $A/\ell^{r}$ contains infinitely many elements of order $\ell^{r}$. 
\end{lemma}

\begin{proof}
Assume that there is an integer $r\geq2$, such that $A/\ell^{r}$ contains only finitely many elements of order $\ell^{r}$. Fix an infinite set $\mathfrak{I}\subset A$ whose elements are linearly independent modulo $\ell$ and denote by $\overline{\mathfrak{I}}\subset A/\ell$ its image. We claim that all but finitely many elements in $\overline{\mathfrak{I}}$ lie in the image of the map $A[\ell^{r}]\to A/\ell$. Indeed, if the image of $x\in\mathfrak{I}$ in $A/\ell^{r}$ is killed by $\ell^{r_{0}}$ with $1\leq r_{0}< r$, then $\ell^{r_{0}}(x-\ell^{r-r_{0}}x')=0\in A$ for some $x'\in A$ and in particular $\tilde{x}:=x-\ell^{r-r_{0}}x'\in A[\ell^{r}]$. Since $r-r_{0}>0$, we find that the reduction of $\tilde{x}$ modulo $\ell$ coincides with the one of $x$. The claim follows as only finitely many images of elements of $\mathfrak{I}$ in $A/\ell^{r}$ have order $\ell^{r}$.\par To conclude the argument note that $A[\ell^{r}]$ is finite and that this in turn contradicts the infiniteness of $\overline{\mathfrak{I}}$. The proof of the Lemma \ref{lem:elementary} is complete.
\end{proof}

\section{A cyclic quotient of a complete intersection}\label{sec:cyclic-quotient}
In this section, we introduce the surfaces that will replace the Enriques surfaces that have seen applications in \cite{Sch-griffiths}.\par Let $k$ be an algebraically closed field of characteristic zero. Let $n\geq 2$ be an integer and fix a primitive $n$-th root of unity $1\neq\omega\in k$. Consider the automorphism $\varphi_{n}$ on $\mathbb{P}^{5}_{k}$ defined by 
\begin{align}\label{eq:varphi}(x_{0} : x_{1} : x_{2} : x_{3} : x_{4} : x_{5})\longmapsto (x_{0} : \omega x_{1} : x_{2} : \omega x_{3} : x_{4} :\omega x_{5}).\end{align} Its fixed locus consists of two disjoint planes given by \begin{align}\label{eq:fixed-locus}
    P_{1}:=\{x_{0}=x_{2}=x_{4}=0\}\ \text{and} \ P_{2}:=\{x_{1}=x_{3}=x_{5}=0\}.
\end{align}

\par The complete linear system $|\mathcal{O}_{\mathbb{P}^{5}_{k}}(n)|$ that parametrizes hypersurfaces $H\subset\mathbb{P}^{5}_{k}$ of degree $n$, splits under $\varphi_{n}$ into $n$ disjoint eigenspaces corresponding to the eigenvalues $\omega^{i}$, $i=0,1,\ldots,n-1$. An easy check shows that the ones that are correlated with the eigenvalue $1$ form projectively a $d$-dimensional family, where $d:=(n+1)(n+2)-1$. We denote this linear system by ${|\mathcal{O}_{\mathbb{P}^{5}_{k}}(n)|}^{\varphi_{n}}$ and call its elements \textit{invariant hypersurfaces}.\par Note that the above is a base-point free linear system, since the sections $x_{0}^{n},x_{1}^{n},\ldots,x_{5}^{n}$ are all invariant. It follows by Bertini's theorem that for general choices of hypersurfaces $H_{1},H_{2},H_{3}\in {|\mathcal{O}_{\mathbb{P}^{5}_{k}}(n)|}^{\varphi_{n}}$, the complete intersection $Y:=H_{1}\cap H_{2}\cap H_{3}$ is smooth and $\varphi_{n}$ acts freely on $Y$.\par A concrete example that works for any $n$ is given by the $3$ equations \begin{align*}\label{concrete-example}\tag{3.3}
         x_{0}^{n}+x_{1}^{n}+x_{3}^{n}+x_{4}^{n}+x_{1}x_{5}^{n-1}+x_{3}x_{5}^{n-1}=0,\\ 2x_{0}^{n}+x_{1}^{n}+3x_{2}^{n}-x_{3}^{n}+5x_{4}^{n}+x_{1}x_{5}^{n-1}-x_{3}x_{5}^{n-1}=0,\\ x_{0}^{n}-3x_{1}^{n}+x_{2}^{n}+x_{3}^{n}-2x_{4}^{n}+7x_{5}^{n}=0.
\end{align*}

 \par We let $S:=Y/\varphi_{n}$ be the $\Z/n$-quotient of $Y$ by the automorphism $\varphi_{n}$. This is a smooth projective surface because the quotient map $\pi\colon Y\to S$ is finite and \'etale. 
 
 \par Now take $k=\C$. If $n=2$, then $S$ is an Enriques surface and $Y$ is its canonical double cover, which is a $K3$ surface. For $n\geq 3$, the surface $S$ is always of general type with $q(S):=h^{0,1}(S)=0$ and has topological fundamental group $\pi_{1}(S)\cong\Z/n$. For the sake of completeness, in the following proposition we compute some of its invariants:
 
 \begin{proposition}\label{prp:invariants-X} Let $n\geq2$ be an integer. Let $Y\subset\mathbb{P}^{5}_{\C}$ be a smooth complete intersection of three invariant hypersurfaces $H_{i}\in{|\mathcal{O}_{\mathbb{P}^{5}_{\C}}(n)|}^{\varphi_{n}}$, such that $\varphi_{n}$ from \eqref{eq:varphi} acts freely on $Y$ and set $S:=Y/\varphi_{n}$. Then the following assertions are true:\begin{enumerate}
     \item $S$ is a smooth projective surface of general type if $n\geq 3$.
     \item The topological fundamental group $\pi_{1}(S)$ is cyclic of order $n$. 
     \item The Hodge numbers of $S$ are given by $h^{2,0}=\frac{n^{2}}{4}(5n^{2}-18n+17)-1$, $h^{1,1}=\frac{n^{2}}{2}(7n^{2}-18n+13),$ and $h^{1,0}=0$.
 \end{enumerate}\end{proposition}
 
\begin{proof} The computations are straightforward. Recall that any smooth complete intersection of dimension $2$ is simply connected. Hence $\pi_{1}(S)\cong\Z/n$ and $q(S)=\frac{1}{2}b_{1}(S)=0$. \par The geometric genus of $S$ is computed from the Euler characteristic formula $\chi(\mathcal{O}_{Y})=n\chi(\mathcal{O}_{S})$, where we use $\chi(\mathcal{O}_{Y})=\frac{n^{3}}{4}(5n^{2}-18n+17)$ (cf. \cite[Proposition V.2.1]{BHPV}).\par The Hirzebruch-Riemann-Roch theorem yields $h^{1,1}(S)=10\chi(\mathcal{O}_{S})- K_{S}^{2}$. We note that $K_{Y}^{2}=n K_{S}^{2}$, whereas $K_{Y}^{2}=9n^{3}(n-2)^{2}$ (see \cite[Proposition V.2.1]{BHPV}). An obvious substitution then gives the desired relation for $h^{1,1}(S)$.\par It remains to check that $S$ is indeed a surface of general type as long as $n\geq 3$. Recall that $K_{Y}=\mathcal{O}_{Y}(3n-6)$ is ample and  $K_{Y}=\pi^{*}(K_{S}\otimes \mathcal{L}^{\otimes n-1})$ (see \cite[Lemma I.17.1]{BHPV}) for some generator $\mathcal{L}\in\Pic(S)[n]\cong\Z/n$. We thus find by the projection formula that $\pi^{*}C.K_{Y}=\pi^{*}C.\pi^{*}(K_{S}\otimes \mathcal{L}^{\otimes n-1})=nC.K_{S}>0$ for all irreducible curves $C\subset S$. In particular $K_{S}$ is ample. This concludes the proof.\end{proof}
\begin{remark} We note that any \'etale $\Z/n$-quotient of a smooth complete intersection of multidegree $(n,n,n)$ in $\mathbb{P}^{5}_{\C}$ has the same invariants with the surface of Proposition \ref{prp:invariants-X}.\end{remark}

\section{Degeneration} Let $k$ be an algebraically closed field of characteristic zero and fix an integer $n\geq 2$. Let $Y\subset \mathbb{P}^{5}_{k}$ be a smooth complete intersection of three very general invariant hypersurfaces $H_{i}\in|\mathcal{O}_{\mathbb{P}^{5}_{k}}(n)|^{\varphi_{n}}$ and set $S:=Y/\varphi_{n}$.\par We aim to exhibit an example of (semi-stable) degenerations for the surface $S$ that will satisfy the properties of Theorem \ref{thm:extend-Brauer-class}. The main result of this section is as follows. 

\begin{theorem}\label{thm:degeneration-X_{1,2,3}} Let $\kappa$ be an algebraically closed field of characteristic zero and fix an integer $n\geq 2$. Let $k$ denote an algebraic closure of the fraction field $\kappa((t))$. There is a regular flat projective scheme $\mathcal{S}\to\Spec\kappa[[t]]$, such that:\begin{enumerate}
    \item the geometric generic fibre $S_{\bar{\eta}}$ is a quotient of the form $Y/\varphi_{n}$, where $Y\subset\mathbb{P}^{5}_{k}$ is a smooth complete intersection of three very general invariant hypersurfaces $H_{i}\in|\mathcal{O}_{\mathbb{P}^{5}_{k}}(n)|^{\varphi_{n}}$ and $\varphi_{n}$ is the automorphism \eqref{eq:varphi}.
    \item the special fibre $S_{0}=V\cup W$ is a reduced simple normal crossing divisor on $\mathcal{S}$ and consists of two irreducible components, where the surface $V$ is the projective plane $\mathbb{P}^{2}_{\kappa}$ and $W$ is a surface with $H^{3}(W,\Z_{\ell}(1))=0$ for all prime factors $\ell$ of $n$. If we set $C:=V\cap W$, then $C\subset V$ is a smooth plane curve of degree $n$. Furthermore, the class $[C]\in \CH^{1}(W)$ is divisible by $n$. 
\end{enumerate}\end{theorem}

We shall need the following lemma. 

\begin{lemma}\label{lem:resolution-Y'}Let $k$ be an algebraically closed field of characteristic zero and fix an integer $n\geq2$. Consider the following three homogeneous polynomials $F_{i}\in k[x_{0},x_{1},x_{2},x_{3},x_{4},x_{5}]$ of degree $n$ \begin{align*}
F_{1}:=\alpha_{0}x_{0}^{n}+\alpha_{1}x_{1}^{n}+\alpha_{2}x_{2}^{n}+\alpha_{3}x_{3}^{n}+\alpha_{4}x_{4}^{n}+x_{5}^{n-1}(x_{1}-x_{3})\\
F_{2}:=\beta_{0}x_{0}^{n}+\beta_{1}x_{1}^{n}+\beta_{3}x_{3}^{n}+\beta_{4}x_{4}^{n}+x_{5}^{n-1}(x_{1}+x_{3})\\
F_{3}:=\gamma_{0}x_{0}^{n}+\gamma_{1}x_{1}^{n}+\gamma_{2}x_{2}^{n}+\gamma_{3}x_{3}^{n}+x_{4}^{n},\end{align*}where $\alpha_{i},\beta_{i},\gamma_{i}\in k$. Set $T:=\{F_{1}=F_{2}=0\}$ and $Y:=\{F_{1}=F_{2}=F_{3}=0\}$ for the corresponding intersections inside $\mathbb{P}^{5}_{k}$. Then there are $\alpha_{i},\beta_{i},\gamma_{i}\in k$, such that the following assertions hold true:\begin{enumerate}
    \item\label{it:smooth-Y_{1,2}} $T$ is a smooth complete intersection of dimension $3$. 
    \item \label{it:Y_{1,2,3}} $Y$ is a reduced and irreducible surface whose singular locus consists only of the point $p_{0}:=(0:0:0:0:0:1)$. In addition, the multiplicity of $Y$ at $p_{0}$ is $n$.
    \item \label{it: blow-up-Y_{1,2,3}} The blow-up $\tilde{Y}:=Bl_{p_{0}}Y$ is smooth and the exceptional divisor $E$ is a smooth plane curve of degree $n$.
    \item\label{it:branched-cover}The resolution $\tilde{Y}$ can be realized as an $(n-1)$-fold cyclic branched cover of a smooth complete intersection $Z\subset\mathbb{P}^{4}_{k}$ of multidegree $(n,n+1)$, branched along $P'+E'$, where $P'$ and $E'$ correspond to $Y\cap\{x_{5}=0\}$ and $E$, respectively.
\end{enumerate}\end{lemma}

\begin{proof} An easy check shows that if we set $\alpha_{i}=1$ for all $i$, then the hypersurface $S:=\{F_{1}=0\}$ is smooth. Let $L\subset|\mathcal{O}_{\mathbb{P}_{k}^{5}}(n)|$ be the linear system spanned by the sections $x_{0}^{n},x_{1}^{n},x_{3}^{n},x_{4}^{n}$ and $x_{5}^{n-1}(x_{1}+x_{3})$. When restricted to $S$, we find that this linear system has only $p_{0}$ as its base point, whereas the image of the corresponding morphism $\varphi_{L}\colon S\setminus\{p_{0}\}\rightarrow\mathbb{P}^{4}_{k}$ has dimension $>1$. Hence Bertini's theorem implies that for general $\beta_{i}\in k$, the intersection $T$ is reduced, irreducible and can only happen to have singularities at $p_{0}$. Note however that $T$ is always smooth at $p_{0}$. Indeed, the tangent space of $T$ at $p_{0}$ is given by $\{x_{1}=x_{3}=0\}\subset k^{5}$, which is $3$-dimensional. This proves \eqref{it:smooth-Y_{1,2}}.

\par We prove \eqref{it:Y_{1,2,3}} and \eqref{it: blow-up-Y_{1,2,3}} together. A Bertini argument as above yields $\gamma_{i}\in k$, such that $Y$ is an irreducible and reduced complete intersection of dimension $2$ that acquires a unique singularity at the point $p_{0}$.

\par Next we compute the blow-up $\pi\colon Bl_{p_{0}}Y\to Y$. We work over the chart $\{x_{5}\neq0\}$ and by abuse of notation we may identify $x_{i}\leftrightarrow \frac{x_{i}}{x_{5}}$. Consider the open sets $U_{i}:=\{(\xi_{0}:\xi_{1}:\xi_{2}:\xi_{3}:\xi_{4})\in\mathbb{P}^{4}_{k}\ |\ \xi_{i}\neq0\}$ and denote by $V_{i}\subset U_{i}\times \mathbb{A}^{1}_{k}$ the closed subset defined by the equations \begin{align}
    \{\zeta^{n-1}(\xi_{0}^{n}+\xi_{1}^{n}+\xi_{2}^{n}+\xi_{3}^{n}+\xi_{4}^{n})+\xi_{i}^{n-1}(\xi_{1}-\xi_{3})=0\}\\ \{\zeta^{n-1}(\beta_{0}\xi_{0}^{n}+\beta_{1}\xi_{1}^{n}+\beta_{3}\xi_{3}^{n}+\beta_{4}\xi_{4}^{n})+\xi_{i}^{n-1}(\xi_{1}+\xi_{3})=0\}\\ \{\gamma_{0}\xi_{0}^{n}+\gamma_{1}\xi_{1}^{n}+\gamma_{2}\xi_{2}^{n}+\gamma_{3}\xi_{3}^{n}+\xi_{4}^{n}=0\}.\end{align} Then the $\{V_{i}\}$ form an open covering of $Bl_{p_{0}}Y$ and the restriction of $\pi$ to $V_{i}$ is given by \begin{align}((\xi_{0}:\xi_{1}:\xi_{2}:\xi_{3}:\xi_{4}),\zeta)\mapsto (\zeta \frac{\xi_{0}}{\xi_{i}},\zeta \frac{\xi_{1}}{\xi_{i}},\zeta \frac{\xi_{2}}{\xi_{i}},\zeta \frac{\xi_{3}}{\xi_{i}},\zeta \frac{\xi_{4}}{\xi_{i}}).
    \end{align} 

Note here that two elements $(p,\zeta_{i})\in V_{i}$ and $(q,\zeta_{j})\in V_{j}$ are identified if and only if $p=q$ and $\zeta_{j}=\frac{\xi_{j}}{\xi_{i}}\zeta_{i}$.

\par A straightforward calculation then shows $E:=\{(\xi_{0}:0:\xi_{2}:0:\xi_{4})\in\mathbb{P}^{4}_{k}\ |\ \gamma_{0}\xi_{0}^{n}+\gamma_{2}\xi_{2}^{n}+\xi_{4}^{n}=0 \}$ is the exceptional divisor of this blow-up. Since $E$ is a smooth (may assume $\gamma_{0}\gamma_{2}\neq0$) Cartier divisor, we find that $Bl_{p_{0}}Y$ is also smooth. The multiplicity of $Y$ at $p_{0}$ coincides with the degree of the obvious embedding $E\subset\mathbb{P}^{2}_{k}$ and thus equals $n$. In particular, $E^{2}=-n$ (cf. \cite{ram}).\par It remains to prove \eqref{it:branched-cover}. It is readily checked that $$V_{i}\subset\{\xi_{0}^{n}+\xi_{1}^{n}+\xi_{2}^{n}+\xi_{3}^{n}+\xi_{4}^{n}\neq0\}\cup \{\beta_{0}\xi_{0}^{n}+\beta_{1}\xi_{1}^{n}+\beta_{3}\xi_{3}^{n}+\beta_{4}\xi_{4}^{n}\neq 0\},$$ as long as $\beta_{4}\gamma_{0}+\gamma_{2}(\beta_{0}-\beta_{4})-\beta_{0}\neq 0$. Over $\{\xi_{0}^{n}+\xi_{1}^{n}+\xi_{2}^{n}+\xi_{3}^{n}+\xi_{4}^{n}\neq0\}$ (similarly over the other open subset), we can rewrite the equations of $V_{i}$ as follows \begin{align}
        \zeta^{n-1}+\frac{\xi_{i}^{n-1}(\xi_{1}-\xi_{3})}{\xi_{0}^{n}+\xi_{1}^{n}+\xi_{2}^{n}+\xi_{3}^{n}+\xi_{4}^{n}}=0\\
        (\xi_{3}-\xi_{1})(\beta_{0}\xi_{0}^{n}+\beta_{1}\xi_{1}^{n}+\beta_{3}\xi_{3}^{n}+\beta_{4}\xi_{4}^{n})+(\xi_{1}+\xi_{3})(\xi_{0}^{n}+\xi_{1}^{n}+\xi_{2}^{n}+\xi_{3}^{n}+\xi_{4}^{n})=0\label{eq:Z_{1,2}_1} \\
        \gamma_{0}\xi_{0}^{n}+\gamma_{1}\xi_{1}^{n}+\gamma_{2}\xi_{2}^{n}+\gamma_{3}\xi_{3}^{n}+\xi_{4}^{n}=0\label{eq:Z_{1,2}-2}.
    \end{align}

\par We let $Z\subset\mathbb{P}^{4}_{k}$ be the complete intersection of multidegree $(n,n+1)$ defined by \eqref{eq:Z_{1,2}_1} and \eqref{eq:Z_{1,2}-2}. Then the natural projection $\pr_{1}\colon Bl_{p_{0}}Y\to\mathbb{P}^{4}_{k}$ factors through $Z$ and the above equations show that over $\pi^{-1}(\{x_{5}\neq0\})$, the map $\pr_{1}$ locally has the form of a cyclic covering of degree $n-1$ whose ramification locus is $E$. Note that $\pr_{1}$ is compatible with $Y\setminus\{p_{0}\}\to Z$ that is induced from the obvious projection $\mathbb{P}^{5}_{k}\setminus\{p_{0}\}\to\mathbb{P}^{4}_{k}$. Similarly, one then shows $Y\setminus\{p_{0}\}\to Z$ is also of the desired form and ramified along $P:=Y\cap\{x_{5}=0\}$. To conclude observe that $Z$ is smooth because its cyclic (branched) cover $Bl_{p_{0}}Y$ is smooth and the ramification locus is a smooth divisor (cf. \cite[Proposition 4.1.6]{laz}). This finishes the proof.\end{proof}

\begin{example}\label{ex:lemma} In the above setting, if we set $\alpha_{0}=2,\alpha_{1}=1,\alpha_{2}=3,\alpha_{3}=-1,\alpha_{4}=5$ and $\beta_{i}=1$ for all $i$, then $T$ is smooth as long as $n$ is even. If $n$ is odd, then $T$ has isolated singularities. The singular locus of $T$ in this case is given by the equations:$$x_{0}=x_{4}=x_{1}+x_{3}=nx_{3}^{n-1}+x_{5}^{n-1}=3nx_{2}^{n}+2(1-n)x_{3}x_{5}^{n-1}=0.$$ A Bertini argument however provides us with $\gamma_{i}\in k$, such that $Y\subset\mathbb{P}^{5}_{k}$ satisfies \eqref{it:Y_{1,2,3}},\eqref{it: blow-up-Y_{1,2,3}} and \eqref{it:branched-cover} from Lemma \ref{lem:resolution-Y'}. For instance, one can take $\gamma_{0}=1,\gamma_{1}=-3,\gamma_{2}=1,\gamma_{3}=1,\gamma_{4}=-2$ if $n\in\{2,3,4,5,6\}$.\end{example}

 \begin{proof}[Proof of Theorem \ref{thm:degeneration-X_{1,2,3}}] We may work over $\C$. Fix $\alpha_{i},\beta_{i},\gamma_{i}\in\C$, so that the consequence of Lemma \ref{lem:resolution-Y'} is true. Then $p_{0}:=(0:0:0:0:0:1)$ is the only fixed point of $\varphi_{n}$ in $Y$. Pick a hypersurface $H:=\{F=0\}\in|\mathcal{O}_{\mathbb{P}^{5}}(n)|^{\varphi_{n}}$ such that the intersection $T\cap H$ is smooth and such that $\varphi_{n}$ acts freely on $T\cap H$.\par Consider the one-parameter family $$\mathcal{Y}:=\{F_{3}+t(F-F_{3})=0\}\subset T\times \mathbb{A}^{1}.$$ We note that $\bar{p}_{0}=(p_{0},0)\in\mathcal{Y}$ is a smooth point, as the tangent space is $\{x_{1}=x_{3}=F(p_{0})t=0\}\subset k^{6}$, which is $3$-dimensional ($F(p_{0})\neq0$). It follows that if we perform a base change with respect to the completion of the local ring at $0\in\A^{1}$, then the resulting projective flat family $\mathcal{Y}\to\Spec\C[[t]]$ is regular and the geometric generic fibre $Y_{\bar{\eta}}$ is a smooth complete intersection of multidegree $(n,n,n)$, such that $\varphi_{n}$ acts freely on it.\par Let $\pi\colon Bl_{\bar{p}_{0}}\mathcal{Y}\to \mathcal{Y}$ be the blow-up at $\bar{p}_{0}\in\mathcal{Y}$ and denote by $\tilde{\varphi}_{n}$ the natural lift of $\varphi_{n}$. Consider the open subset $$U_{i}:=\{(\xi_{j})\in \mathbb{P}^{5}\ |\ \xi_{i}\neq0\}\times \A^{1}\subset Bl_{\bar{p}_{0}}(\mathbb{P}^{5}\times\A^{1}).$$ The exceptional divisor $V$ is a projective plane and as one can easily check it is given by the equations $\{\xi_{1}=\xi_{3}=\xi_{5}=\zeta=0\}\subset U_{i}$ for $i\in\{0,2,4\}$. If $i\in\{0,2,4\}$, then the restriction $\tilde{\varphi}_{n}|_{U_{i}}$ has the form
$$((\xi_{0}:\xi_{1}:\xi_{2}:\xi_{3}:\xi_{4}:\xi_{5}),\zeta)\mapsto((\xi_{0}:\omega\xi_{1}:\xi_{2}:\omega\xi_{3}:\xi_{4}:\omega\xi_{5}),\omega^{-1}\zeta)$$ from which immediately one obtains that $V$ constitutes the fix point locus of $\tilde{\varphi}_{n}$ inside $Bl_{\bar{p}_{0}}\mathcal{Y}$.\par It is clear how to proceed now. Set $\tilde{\mathcal{Y}}:=Bl_{\bar{p}_{0}}\mathcal{Y}$ and note by \eqref{it:Y_{1,2,3}} that the component $V$ appears with multiplicity $n$ in the special fibre of $\tilde{\mathcal{Y}}\to\Spec\C[[t]]$. Since $\Fix(\tilde{\varphi}_{n})=V$, quotienting out by $\tilde{\varphi}_{n}$ yields a regular flat projective family $\mathcal{S}:=\tilde{\mathcal{Y}}/\tilde{\varphi}_{n}\to\Spec\C[[t]]$ with geometric generic fibre $S_{\bar{\eta}}=Y_{\bar{\eta}}/\varphi_{n}$ degenerating into two components that are both clearly reduced. By abuse of notation we let $V\subset S_{0}$ denote the component corresponding to the branch locus of $\tilde{\mathcal{Y}}\to\mathcal{S}$. The other component $W$ of the special fibre $S_{0}$ is the quotient surface $\tilde{Y}/\tilde{\varphi}_{n}$, which is smooth by \eqref{it: blow-up-Y_{1,2,3}}. Note that the assertion $C:=V\cap W$ is a smooth plane curve of degree $n$ in $V$ follows from \eqref{it: blow-up-Y_{1,2,3}}, as well. Moreover, the class of $C$ in $\CH^1(W)$ is divisible by $n$, as $C\subset W$ is the branch locus of the $n$-fold cyclic cover $\tilde{Y}\to W$.

\par It remains to check $H^{3}(W,\Z_{\ell}(1))=0$ for all prime factors $\ell$ of $n$. From \cite[Proposition 2.5.7]{persson}, we get $\coker(H_{1}(C,\Z)\to H_{1}(W,\Z))=0$. By Poincar\'e duality, the latter is of course equivalent with $\iota_{*}\colon H^{1}_{sing}(C,\Z(0))\to H^{3}_{sing}(W,\Z(1))$ being surjective, where $\iota\colon C\hookrightarrow W$ is the inclusion. Especially, the comparison theorem (see \cite[Theorem III.3.12]{milne}) implies $\iota_{*}\colon H^{1}(C,\Z_{\ell}(0))\to H^{3}(W,\Z_{\ell}(1))$ is surjective for all primes $\ell$. Thus, our task is reduced to show that the push-forward $\iota_{*}$ is the zero map for all prime factors $\ell$ of $n$. In what follows, we shall identify the branch locus $C\subset W$ of the cyclic cover $\tilde{Y}\to W$ with its ramification locus $E$. By \cite[Lemma 2.3.6]{CS-brauer}, we find the commutative diagram \begin{center}
     \begin{tikzcd}
{H^{1}(C,\Z_{\ell}(0))} \arrow[d, "\times (n-1)"] \arrow[r, "j'_{*}"] & {H^{3}(Z,\Z_{\ell}(1))} \arrow[d, "f^{*}"] \\
{H^{1}(C,\Z_{\ell}(0))} \arrow[r, "j_{*}"]                            & {H^{3}(\tilde{Y},\Z_{\ell}(1))},       
\end{tikzcd}
 \end{center} where we note that $H^{3}(Z,\Z_{\ell}(1))=0$ because the surface $Z$ is a smooth complete intersection in $\mathbb{P}^{4}_{\C}$ as well as $\times(n-1)\colon H^{1}(C,\Z_{\ell}(0))\to H^{1}(C,\Z_{\ell}(0))$ is an isomorphism, since the integer $n-1$ is invertible in $\Z_{\ell}$. Consequently $j_{*}\colon H^{1}(C,\Z_{\ell}(0))\to H^{3}(\tilde{Y},\Z_{\ell}(1))$ is the zero map. To conclude observe that $\iota_{*}\colon H^{1}(C,\Z_{\ell}(0))\to H^{3}(W,\Z_{\ell}(1))$ factors through $j_{*}$. The proof is complete.
 \end{proof}
 
 For the sake of completeness, let us compute the geometric genus of $W$.\begin{proposition}\label{prop:geometric-genus-W} In the notation of Theorem \ref{thm:degeneration-X_{1,2,3}}, the geometric genus of $W$ is given by the formula $$g(W)=\frac{n^{2}}{4}(5n^{2}-18n+15)+\frac{3n}{2}-2.$$\end{proposition}
 \begin{proof} By Proposition \ref{prp:invariants-X}, it suffices to show $g(S_{\bar{\eta}})=g(C)+g(W)$, where $S_{\bar{\eta}}$ is the geometric generic fibre of the degeneration $\mathcal{S}\to\Spec\C[[t]]$ from Theorem \ref{thm:degeneration-X_{1,2,3}}. To this end, consider the morphism $\pi\colon W\to W'$ obtained by contracting the $-n^{2}$ curve $C$. In the notation of Lemma \ref{lem:resolution-Y'}, the surface $W'$ is the quotient $Y/\varphi_{n}$. We claim that the singular point $p_{0}\in W'$ is Du Bois, i.e. the canonical map \begin{align}\label{du-bois}R^{i}\pi_{*}\mathcal{O}_{W}\to H^{i}(C,\mathcal{O}_{C})\end{align} is isomorphism for $i>0$. Since $\pi$ is isomorphism outside $p_{0}$, we find that the coherent sheaf $\mathcal{F}^{i}:=R^{i}\pi_{*}\mathcal{O}_{W}$ is supported at $p_{0}$ if $i>0$ and so $\mathcal{F}^{i}=\hat{\mathcal{F}}^{i}$, where $\hat{\mathcal{F}}^{i}$ is the completion of the stalk at $p_{0}$ and $i>0$. Let $C_{m}$ denote the closed subscheme of $W$ defined by the ideal sheaf $\mathfrak{J}^{m}:=\mathcal{O}_{W}(-mC)$. The theorem on Formal functions \cite[Theorem III.11.1]{hart} then yields $$\hat{\mathcal{F}}^{i}\cong\lim_{\longleftarrow}H^{i}(C_{m},\mathcal{O}_{C_{m}}).$$ There are canonical short exact sequences $$0\longrightarrow\mathfrak{J}^{m}/\mathfrak{J}^{m+1}\longrightarrow\mathcal{O}_{C_{m+1}}\longrightarrow\mathcal{O}_{C_{m}}\longrightarrow0$$ for each $m$. We have $\mathfrak{J}^{m}/\mathfrak{J}^{m+1}\cong\mathcal{N}_{C/W}^{\otimes -m}$, where $\mathcal{N}_{C/W}$ is the normal bundle. It follows by Serre duality $$H^{1}(C,\mathfrak{J}^{m}/\mathfrak{J}^{m+1})\cong H^{0}(C,\mathcal{N}_{C/W}^{\otimes m}\otimes\omega_{C})=0,$$ where the last term vanishes because $\deg_{C}(\mathcal{N}_{C/W}^{\otimes m}\otimes\omega_{C})=-mn^{2}+(n-1)(n-2)-2<0$. The long exact sequence of cohomology thus gives $H^{i}(C_{m},\mathcal{O}_{C_{m}})\cong H^{i}(C,\mathcal{O}_{C})$ for all $m\geq 1$ and $i\geq 1$. In particular, $\hat{\mathcal{F}}^{1}\cong H^{1}(C,\mathcal{O}_{C})$ and $\hat{\mathcal{F}}^{i}\cong 0$ for $i>2$. This proves \eqref{du-bois}.\par It is easy to see that the Leray spectral sequence $$E^{i,j}_{2}:=H^{i}(W',R^{j}\pi_{*}\mathcal{O}_{W})\implies H^{i+j}(W,\mathcal{O}_{W})$$ yields the exact sequence 
 \begin{align}\label{spectral-sequence}0\to H^{1}(W',\mathcal{O}_{W'})\to H^{1}(W,\mathcal{O}_{W})\to H^{1}(C,\mathcal{O}_{C})\to H^{2}(W',\mathcal{O}_{W'})\to H^{2}(W,\mathcal{O}_{W})\to0.\end{align} Recall by Theorem \ref{thm:degeneration-X_{1,2,3}} that the torsion free part of $H^{3}_{sing}(W,\Z(1))$ is zero and thus, Poincar\'e duality gives that the first Betti number $b_{1}$ of $W$ is zero. In particular, we have $H^{1}(W,\mathcal{O}_{W})=0$. In addition, by \cite[Theorem 1 and Example (b)]{St}, we find that $H^{i}(W',\mathcal{O}_{W'})\cong H^{i}(S_{\bar{\eta}},\mathcal{O}_{S_{\bar{\eta}}})$. Altogether, from \eqref{spectral-sequence}, we obtain $g(S_{\bar{\eta}})=g(W)+g(C),$ as we wanted. The proof is complete.\end{proof}
 
\section{Extending Brauer Classes}
In this section, we aim to prove Theorem \ref{thm:extend-Brauer-class}. 

\begin{theorem} \label{thm:extending-brauer-classes-2} Let $\kappa$ be an algebraically closed field of characteristic zero and let $n\geq 2$ be an integer. Write $n=\ell_{1}^{r_{1}}\ell_{2}^{r_{2}}\ldots\ell_{s}^{r_{s}}$ for its prime factorisation. Let $\mathcal{S}\to\Spec\kappa[[t]]$ denote the degeneration from Theorem \ref{thm:degeneration-X_{1,2,3}}. After a possible base change, followed by a resolution of $\mathcal{S}$,
the following holds:
\par For all $1\leq i\leq s$, there is a class $\alpha_{i}\in\Br(\mathcal{S})[\ell_{i}^{r_{i}}]$, such that 
\begin{enumerate}
    \item\label{itm:property-ext} $\delta(\alpha_{i}|_{S_{\bar{\eta}}})\in H^{3}(S_{\bar{\eta}},\Z_{\ell_{i}}(1))\cong\Z/\ell_{i}^{r_{i}}$ is a generator, where $\delta$ is the Bockstein map defined on the Brauer group $\Br(S_{\bar{\eta}})$ and such that 
    \item\label{itm:property-rest-zero} for every component $S_{0j}$ of the special fibre, $\alpha_{i}|_{S_{0j}}=0\in\Br(S_{0j})$.\end{enumerate}\end{theorem}
    
\begin{remark}\label{rem:base-change} Let $\mathcal{S}\to\Spec R$ be a proper strictly semi-stable scheme over a discrete valuation ring $R$ with geometrically connected fibres of relative dimension $2$. When performing a finite ramified base change $\tilde{R}/R$, the model $\mathcal{S}_{\tilde{R}}:=\mathcal{S}\times_{R}\tilde{R}$ becomes singular. Following Hartl (proof of \cite[Proposition 2.2]{hartl}), one finds that the family $\mathcal{S}_{\tilde{R}}\to\Spec\tilde{R}$ can be made again into a strictly semi-stable $\mathcal{\tilde{S}}\to\Spec\tilde{R}$ by repeatedly blowing up all non-Cartier divisors of the special fibre. In particular, if $S_{0}=\bigcup_{i=1}^{m}S_{0i}$ is a chain, then so is the special fibre \begin{align}\label{special-fibre}
    \tilde{S}_{0}=S_{01}\cup (\bigcup_{j=1}^{t}R_{1,j})\cup S_{02}\cup (\bigcup_{j=1}^{t}R_{2,j})\cup \dots\cup S_{0m}\end{align} of the resulting model. The new components $R_{i,j}$ are all minimal ruled surfaces over $C_{i,i+1}:=S_{0i}\cap S_{0i+1}$ and we shall denote by $\pi_{i,j}\colon R_{i,j}\to C_{i,i+1}$ the corresponding projection. Moreover, the double intersections $S_{0i}\cap R_{i,1}$, $R_{i,t}\cap S_{0i+1}$ and $R_{i,j}\cap R_{i,j+1}$ are isomorphic to $C_{i,i+1}$ for all $1\leq i< m$ and all $1\leq j< t$, i.e., constitute sections.\end{remark}
\begin{proof}[Proof of Theorem \ref{thm:extending-brauer-classes-2}] Write $S:=S_{\eta}$ for the generic fibre. Recall by \cite[Section 2.2.2]{CS-brauer} $$\Br(S_{\bar{\eta}})=\lim_{\substack{\longrightarrow\\ \kappa((t))\subset K}}\Br(S_{K}),$$ where the direct limit runs over all finite field extensions $K/\kappa((t))$. Thus, we may find a finite extension $K$ of $\kappa((t))$, such that $\Br(S_{K})[n]\twoheadrightarrow \Br(S_{\bar{\eta}})[n]$ is surjective. Let $R\subset K$ be the normalization of $\kappa[[t]]$ in $K$. Then $R$ is a henselian discrete valuation ring that is finite over $\kappa[[t]]$. Consider the base change $\mathcal{S}_{R}\to\Spec R$. By \cite[Proposition 2.2]{hartl}, the family $\mathcal{S}_{R}\to\Spec R$ can be turned into strictly semi-stable by repeatedly blowing up all non Cartier components of the special fibre. Let $\tilde{\mathcal{S}}\to\Spec R$ denote the induced semi-stable model. The completion of $R$ is isomorphic to a formal power series ring over $\kappa$. Fix an isomorphism $\hat{R}\cong\kappa[[t']]$ and observe that the model $\tilde{\mathcal{S}}_{\kappa[[t']]}\to\Spec\kappa[[t']]$ remains strictly semi-stable (see \cite[Lemma 54.11.2]{stacks-project}). The invariance of \'etale cohomology under extensions of seperably closed ground fields then yields $\Br(S_{k'})[n]\twoheadrightarrow \Br(S_{\bar{k'}})[n]$ is also surjective, where $k':=\kappa((t'))$.\par It follows from the above discussion that we may assume $\Br(S_{\eta})[n]\twoheadrightarrow\Br(S_{\bar{\eta}})[n]$ is surjective. Note however that the special fibre $S_{0}$ takes the form $S_{0}=W\cup (\bigcup_{j=1}^{N} R_{j})\cup V$, where $N\geq 1$ and where the notation is consistent with \eqref{special-fibre}.

\par Pick a class $\alpha_{i}\in\Br(S_{\bar{\eta}})[\ell^{r_{i}}_{i}]$, such that $\delta(\alpha_{i})\in H^{3}(S_{\bar{\eta}},\Z_{\ell_{i}}(1))\cong\Z/\ell_{i}^{r_{i}}$ is a generator and denote by the same symbol a lift of this class on $\Br(S_{\eta})[\ell_{i}^{r_{i}}]$. We wish to show that $\alpha_{i}\in \Br(S_{\eta})[\ell_{i}^{r_{i}}]$ extends to an honest class on $\Br(\mathcal{S})[\ell_{i}^{r_{i}}]$ after a possible base change followed by a resolution of $\mathcal{S}$. As a first step, we prove the following:
\begin{claim}\label{claim-1} We may assume after a possible base change followed by a resolution of $\mathcal{S}$ that $\alpha_{i}\in\Br(S_{\eta})[\ell^{r_{i}}_{i}]$ extends to a class $\tilde{\alpha}_{i}\in\Br(\mathcal{S}\setminus \bigcup_{i=1}^{N}R_{j})[\ell_{i}^{r_{i}}]$, where we can take $N=dn^{r}-1$ for some integers $r\geq 2$ and $d>0$.\end{claim}

\begin{proof}[Proof of Claim 1] We perform the $n^{2}:1$ base change $t\mapsto t^{n^{2}}$, followed by resolving the singularities as above and let $f\colon\tilde{\mathcal{S}}\to\mathcal{S}_{\kappa[[t]]}\to\mathcal{S}$ be the induced semi-stable model. Then the special fibre $\tilde{S}_{0}$ has the form (see \eqref{special-fibre})\begin{align}
    \tilde{S}_{0}=W\cup(\bigcup_{j=1}^{n^{2}-1}R_{1,j})\cup R_{1}\cup (\bigcup_{j=1}^{n^{2}-1}R_{2,j})\cup R_{2} \cup \dots \cup R_{N}\cup (\bigcup_{j=1}^{n^{2}-1}R_{N +1,j})\cup V,\end{align} and the number of ruled components has increased to $N':=N + (N+1)(n^{2}-1)=n^{2}(N+1)-1$. We set $U:=\mathcal{S}\setminus \bigcup_{i=1}^{N}R_{j}$ and let $\tilde{U}:=f^{-1}(U)$ be the open subset obtained by removing the ruled components of $\tilde{S}_{0}$ from $\tilde{\mathcal{S}}$. By \cite[Theorem 3.7.5]{CS-brauer}, there is a commutative diagram\begin{center}
  \begin{tikzcd}
0 \arrow[r] & {\Br(U)[\ell_{i}^{r_{i}}]} \arrow[r] \arrow[d, "f^{*}"] & {\Br(S_{\eta})[\ell_{i}^{r_{i}}]} \arrow[r, "\partial"] \arrow[d, "f^{*}"] & {H^{1}(\kappa(V),\Z/\ell_{i}^{r_{i}})\times H^{1}(\kappa(W),\Z/\ell_{i}^{r_{i}}) } \arrow[d, "n^{2}\cdot f^{*}=0"] \\
0 \arrow[r] & {\Br(\tilde{U})[\ell_{i}^{r_{i}}]} \arrow[r]            & {\Br(\tilde{S}_{\eta})[\ell_{i}^{r_{i}}]} \arrow[r, "\partial"]            & {H^{1}(\kappa(V),\Z/\ell_{i}^{r_{i}})\times H^{1}(\kappa(W),\Z/\ell_{i}^{r_{i}}) },                      \end{tikzcd}\end{center} where the rows are exact (see \cite[Theorem 3.7.1]{CS-brauer}). It follows that the class $\tilde{\alpha}_{i}:=f^{*}\alpha_{i}\in\Br(\tilde{S}_{\eta})[\ell_{i}^{r_{i}}]$ extends on $\tilde{U}$. This proves the claim.\end{proof} Let $R:=\bigcup_{j=1}^{N}R_{j}\subset\mathcal{S}$ be the union of ruled components in the special fibre $S_{0}$ and consider the class $\bar{\alpha}_{i}:=\delta(\tilde{\alpha}_{i})\in H^{3}_{cont}(\mathcal{S}\setminus R,\Z_{\ell_{i}}(1))[\ell_{i}^{r_{i}}]$. We aim to show that $\bar{\alpha}_{i}$ lifts to a torsion class on $H^{3}_{cont}(\mathcal{S},\Z_{\ell_{i}}(1))$. To this end, consider the Gysin sequence \begin{align}\label{eq:Gysin-sequence-H^{3}}
    H^{1}_{BM}(R,\Z_{\ell_{i}}(0))\overset{\iota_{*}}{\to}H^{3}_{cont}(\mathcal{S},\Z_{\ell_{i}}(1))\to H^{3}_{cont}(\mathcal{S}\setminus R,\Z_{\ell_{i}}(1))\overset{\partial}{\to}H^{2}_{BM}(R,\Z_{\ell_{i}}(0))
\end{align}\begin{claim}\label{claim-2} The group $H^{2}_{BM}(R,\Z_{\ell_{i}}(0))$ is torsion free. Moreover, the push-forward maps $(\iota_{j})_{*}\colon H^{1}(R_{j},A(0))\to H^{1}_{BM}(R,A(0))$, where $A\in\{\Z/\ell_{i}^{m},\Z_{\ell_{i}},\Q_{\ell_{i}},\Q_{\ell_{i}}/\Z_{\ell_{i}}\}$ yield an isomorphism \begin{align}\label{eq:isomorphism-R}
    \bigoplus_{j=1}^{N} H^{1}(R_{j},A(0))\overset{\cong}{\to}H^{1}_{BM}(R,A(0)).\end{align}\end{claim} \begin{proof}[Proof of Claim 2] We first prove \eqref{eq:isomorphism-R}. We proceed by induction on $N$. Clearly the isomorphism \eqref{eq:isomorphism-R} holds if $N=1$. Suppose $N\geq 2$ and set $R^{N-1}:=\bigcup_{j=1}^{N-1}R_{j}$. We claim that the push-forward map $\iota_{*}\colon H^{1}(R^{N-1},A(0))\to H^{1}(R,A(0))$ is injective. Indeed, by \cite[Definition 4.2 (P2)]{Sch-refined}, there is a commutative diagram \begin{center}\begin{tikzcd}
{H^{0}(R,A(0))} \arrow[r, two heads]                                            & {H^{0}(R_{N}\setminus C,A(0))} \arrow[r]                          & {H^{1}(R^{N-1},A(0))} \arrow[r, "\iota_{*}"] & {H^{1}(R,A(0))} \\
{A=H^{0}(R_{N},A(0))} \arrow[u, "(\iota_{N})_{*}"] \arrow[r, "\id"] & {{H^{0}(R_{N}\setminus C,A(0))=A},} \arrow[u, "\id"', shift left] &                                                    &                
\end{tikzcd}\end{center}
where the upper row is the Gysin sequence of the pair $(R,R^{N-1})$ and where $C:=C_{N-1,N}$. Exactness of this sequence then yields $\iota_{*}\colon H^{1}(R^{N-1},A(0))\to H^{1}(R,A(0))$ is injective.\par Consider the commutative diagram \begin{center}
    \begin{tikzcd}
0 \arrow[r] & {H^{1}(R^{N-1},A(0))} \arrow[r] & {H^{1}(R,A(0))} \arrow[r]                                                                     & {H^{1}(R_{N}\setminus C,A(0))}                  \\
            & 0 \arrow[u] \arrow[r]                 & {H^{1}(R_{N},A(0))} \arrow[u, "(\iota_{N})_{*}"', shift left] \arrow[r] \arrow[r] & {H^{1}(R_{N}\setminus C,A(0))}, \arrow[u, "\id"]
\end{tikzcd}
\end{center}
where the two rows are induced by the Gysin sequences of the pairs $(R,R^{N-1})$ and $(R_{N},C)$, respectively. To deduce \eqref{eq:isomorphism-R} thus remains to prove $H^{1}(R_{N},A(0))\to H^{1}(R_{N}\setminus C,A(0))$ is surjective. By exactness of the Gysin sequence of $(R_{N},C)$, it suffices to show $H^{0}(C,A(-1))\to H^{2}(R_{N},A(0))$ is injective. To see this just note that the composite $$H^{j}(C,A(m))\overset{s_{*}}\to H^{j+2}(R_{N},A(m+1))\overset{(\pi_{N})_{*}}{\to}H^{j}(C,A(m))$$ is the identity.\par Next we prove that $H^{2}_{BM}(R,\Z_{\ell_{i}}(0))$ is torsion free. In the above argument we have already seen that $H^{1}(R,\Z_{\ell_{i}}(0))\to H^{1}(R_{N}\setminus C,\Z_{\ell_{i}}(0))$ is surjective. Hence exactness of the Gysin sequence of $(R,R^{N-1})$ gives that $\iota_{*}\colon H^{2}(R^{N-1},\Z_{\ell_{i}}(0))\to H^{2}(R,\Z_{\ell_{i}}(0))$ is injective. As before we have a commutative diagram \begin{center}
   \begin{tikzcd}
0 \arrow[r] & {H^{2}(R^{N-1},\Z_{\ell_{i}}(0))} \arrow[r] & {H^{2}(R,\Z_{\ell_{i}}(0))} \arrow[r]                                                                                & {H^{2}(R_{N}\setminus C,\Z_{\ell_{i}}(0))}                      \\
            &                                                   & {H^{2}(R_{N},\Z_{\ell_{i}}(0))} \arrow[u, "(\iota_{N})_{*}"', shift left] \arrow[r] \arrow[r, two heads] & {{H^{2}(R_{N}\setminus C,\Z_{\ell_{i}}(0))},} \arrow[u, "\id"]
\end{tikzcd}
\end{center}
where the surjectivity of $H^{2}(R_{N},\Z_{\ell_{i}}(0))\to H^{2}(R_{N}\setminus C,\Z_{\ell_{i}}(0))$ follows again from the exactness of the Gysin sequence of $(R_{N},C)$ (and where we use that $H^{1}(C,\Z_{\ell_{i}}(-1))\to H^{3}(R_{N},\Z_{\ell_{i}}(0))$ is injective). Thus we obtain a short exact sequence $$0\to H^{2}(R^{N-1},\Z_{\ell_{i}}(0))\to H^{2}(R,\Z_{\ell_{i}}(0))\to H^{2}(R_{N}\setminus C,\Z_{\ell_{i}}(0))\to 0,$$ where $H^{2}(R^{N-1},\Z_{\ell_{i}}(0))$ is torsion free by the induction hypothesis and $H^{2}(R_{N}\setminus C,\Z_{\ell_{i}}(0))\cong \Z_{\ell_{i}}\cdot f$ is free of rank $1$, generated by the class of a fibre. This concludes the proof.\end{proof}

Since $H^{2}_{BM}(R,\Z_{\ell_{i}}(0))$ is torsion free by Claim \ref{claim-2}, we find that the residue $\partial\bar{\alpha}_{i}=0$ vanishes and so the class $\bar{\alpha}_{i}$ extends to $H^{3}_{cont}(\mathcal{S},\Z_{\ell_{i}}(1))$ (see \eqref{eq:Gysin-sequence-H^{3}}). The technical part of the proof is then to show that we can choose the lift $\bar{\alpha}_{i}\in H^{3}_{cont}(\mathcal{S},\Z_{\ell_{i}}(1))$ to be torsion. With this goal in mind, let us compute the torsion in $H^{3}_{cont}(\mathcal{S},\Z_{\ell_{i}}(1))$.

\begin{claim}\label{claim-3} There is a canonical extension \begin{align}\label{canonical-extension}
0\to\Z/\ell_{i}^{r_{i}}\to H^{3}_{cont}(\mathcal{S},\Z_{\ell_{i}}(1))\overset{\{\iota_{j}^{*}\}_{j}}{\to}\bigoplus_{j=1}^{N}H^{3}_{cont}(R_{j},\Z_{\ell_{i}}(1))\to 0.\end{align} In particular, $H_{cont}^{3}(\mathcal{S},\Z_{\ell_{i}}(1))_{\tors}\cong\Z/\ell_{i}^{r_{i}}$.\end{claim}

\begin{proof} The proper base change theorem (see \cite[Corollary 3.4]{milne}) combined with \cite[(0.2)]{jannsen} yield the canonical isomorphism $H^{3}_{cont}(\mathcal{S},\Z_{\ell_{i}}(1))\cong H^{3}_{cont}(S_{0},\Z_{\ell_{i}}(1))$. The Mayer-Vietoris sequence gives rise to the following short exact sequence\begin{align}\label{MV} 0\to K\to H^{3}_{cont}(S_{0},\Z_{\ell_{i}}(1))\overset{\{\iota_{j}^{*}\}_{j}}{\to} \bigoplus_{j=1}^{N}H^{3}(R_{j},\Z_{\ell_{i}}(1))\to 0,\end{align} where $K$ is defined as the cokernel of the map  \begin{align}\label{A} H^{2}(W,\Z_{\ell_{i}}(1))\oplus(\bigoplus_{j=1}^{N}H^{2}(R_{j},\Z_{\ell_{i}}(1)))\oplus H^{2}(V,\Z_{\ell_{i}}(1))\to \bigoplus_{j=0}^{N} H^{2}(C_{j,j+1},\Z_{\ell_{i}}(1))\cong \Z_{\ell_{i}}^{N+1}\end{align} and where $C_{j,j+1}$ stands for the double intersection of two consecutive components of the special fibre $S_{0}$, i.e. $C_{0,1}:=W\cap R_{1},\ C_{j,j+1}:=R_{j}\cap R_{j+1}$ for $1\leq j<N$ and $C_{N,N+1}:= R_{N}\cap V$.\par To show that $\{\iota_{j}^{*}\}_{j}$ in \eqref{MV} is onto, we use that $H^{3}(C_{j,j+1},\Z_{\ell_{i}}(1))=0$ (see \cite[Theorem VI.1.1]{milne}). The terms $H^{3}(W,\Z_{\ell_{i}}(1))$ and $H^{3}(V,\Z_{\ell_{i}}(1))$ do not appear in \eqref{MV}, as both vanish (see Theorem \ref{thm:degeneration-X_{1,2,3}}).\par It remains to compute $K$. We claim that the summation map $\Z_{\ell_{i}}^{N+1}\to\Z_{\ell_{i}}$ yields an isomorphism $K\cong \Z/\ell_{i}^{r_{i}}$. To this end, recall $H^{2}(R_{j},\Z_{\ell_{i}}(1))=\Z_{\ell_{i}}\cdot\cl^{1}([C_{j,j+1}])\times \Z_{\ell_{i}}\cdot f_{j}$, where $f_{j}$ denotes the cycle class of a fibre of the projection $\pi_{j}\colon R_{j}\to C_{1,2}$. Via \eqref{A} the tuple $(0,\ldots,0,f_{j},0,\ldots,0)$ maps to $$(0,\ldots,0,-f_{j}\cdot C_{j-1,j},f_{j}\cdot C_{j,j+1},0,\ldots,0)=(0,\ldots,0,-1,1,0,\ldots,0).$$ In addition, the self-intersection $C_{j,j+1}^{2}=\pm n^{2}$, i.e., it is a multiple of $\ell_{i}^{r_{i}}$, and $$\gcd (\gamma_{1}\cdot C_{0,1},\gamma_{2}\cdot C_{N,N+1})=\ell_{i}^{r_{i}},$$ where the cycles $\gamma_{1}$ and $\gamma_{2}$ run through $H^{2}(W,\Z_{\ell_{i}}(1))$ and $H^{2}(V,\Z_{\ell_{i}}(1))$, respectively. For the second claim recall that $W$ admits an $n$-fold cyclic covering branched along $C_{0,1}$ and so $C_{0,1}\sim n\cdot D$ for some divisor $D$, whereas $C_{N,N+1}$ sits as a degree $n$ plane curve in $V$. It follows that the summation map induces a well-defined map $K\twoheadrightarrow\Z/\ell_{i}^{r_{i}}$. The injectivity condition can be easily checked and thus details are omitted.\end{proof} 

Fix the projection $\pi_{j}\colon R_{j}\to C$ of the ruled surface $R_{j}$ for all $1\leq j\leq N$ and consider the composite \begin{align}\label{prod}
\Phi\colon H^{1}(C,A(0))^{\oplus N}\overset{\{\pi_{j}^{*}\}}{\to}\bigoplus_{j=1}^{N}H^{1}(R_{j},A(0))\overset{\{\iota_{j}^{*}\}\circ\{\iota_{j*}\}}{\to}\bigoplus_{j=1}^{N}H^{3}(R_{j},A(1))\overset{\{\pi_{j*}\}}{\to} H^{1}(C,A(0))^{\oplus N}.\end{align} We shall prove a series of properties of $\Phi$ that will be needed later.\begin{lemma}\label{properties: Pi} The map $\Phi$ is explicitly given by $$(\alpha_{1},\alpha_{2},\ldots,\alpha_{N-1},\alpha_{N})\mapsto (\alpha_{2}-2\alpha_{1},\alpha_{1}+\alpha_{3}-2\alpha_{2},\ldots,\alpha_{j-1}+\alpha_{j+1}-2\alpha_{j},\ldots,\alpha_{N-1}-2\alpha_{N}).$$\end{lemma}\begin{proof} We shall frequently cite properties of cohomology with support from \cite[Appendix A]{Sch-moving}, which applies to our setting by \eqref{eq:borel-moore-support}. Let $\tilde{\Phi}:=\{\iota_{j}^{*}\}\circ\{\iota_{j*}\}\circ\{\pi_{j}^{*}\}$ and note that \begin{align}\label{tilde-prod}
    \tilde{\Phi}(\alpha_{1},\alpha_{2},\ldots,\alpha_{N-1},\alpha_{N})=(\iota_{1}^{*}(\sum_{j=1}^{N}\iota_{j*}\pi_{j}^{*}\alpha_{j}),\iota_{2}^{*}(\sum_{j=1}^{N}\iota_{j*}\pi_{j}^{*}\alpha_{j}),\ldots,\iota_{N}^{*}(\sum_{j=1}^{N}\iota_{j*}\pi_{j}^{*}\alpha_{j})).\end{align} 
    We compute $\iota_{t}^{*}(\sum_{j=1}^{N}\iota_{j*}\pi_{j}^{*}\alpha_{j})$ for all $1\leq t\leq N$. We have $$\iota_{t}^{*}(\sum_{j=1}^{N}\iota_{j*}\pi_{j}^{*}\alpha_{j})=\sum_{j=1}^{N}(\iota_{j*}\pi_{j}^{*}\alpha_{j}\cup [R_{t}])=\sum_{j=1}^{N}\iota_{j*}(\pi_{j}^{*}\alpha_{j}\cup \iota_{j}^{*}[R_{t}]),$$ where $[R_{t}]\in H^{0}(R_{t},\Z_{\ell_{i}}(0))=H^{2}_{R_{t},cont}(\mathcal{S},\Z_{\ell_{i}}(1))$ denotes the fundamental class and where in the last equality we use the projection formula (cf. \cite[Lemma A.19]{Sch-moving}). Since $R_{t}\cap R_{j}=\emptyset$ for $j\notin\{t-1,t,t+1\}$, we find \begin{align*}
        \iota_{1}^{*}(\sum_{j=1}^{N}\iota_{j*}\pi_{j}^{*}\alpha_{j})=\iota_{1*}(\pi_{1}^{*}\alpha_{1}\cup \iota_{1}^{*}[R_{1}])+\iota_{2*}(\pi_{2}^{*}\alpha_{2}\cup \iota_{2}^{*}[R_{1}])\\
         \iota_{t}^{*}(\sum_{j=1}^{N}\iota_{j*}\pi_{j}^{*}\alpha_{j})=\sum_{j=t-1}^{t+1}\iota_{j*}(\pi_{j}^{*}\alpha_{j}\cup \iota_{j}^{*}[R_{t}])\ \text{for}\ 1<t<N \\\iota_{N}^{*}(\sum_{j=1}^{N}\iota_{j*}\pi_{j}^{*}\alpha_{j})=\iota_{N-1*}(\pi_{N-1}^{*}\alpha_{N-1}\cup \iota_{N-1}^{*}[R_{N}])+\iota_{N*}(\pi_{N}^{*}\alpha_{N}\cup \iota_{N}^{*}[R_{N}]).
    \end{align*} Observe that $\iota_{t-1*}(\pi_{t-1}^{*}\alpha_{t-1}\cup \iota_{t-1}^{*}[R_{t}])=\iota_{t-1*}(\pi_{t-1}^{*}\alpha_{t-1}\cup [C_{t-1,t}])=s^{t}_{t-1*}(\alpha_{t-1})$, where $s^{t}_{t-1}$ denotes the section $C\to C_{t-1,t}\hookrightarrow R_{t}$. Similarly, one obtains $$\iota_{t+1*}(\pi_{t+1}^{*}\alpha_{t+1}\cup\iota_{t+1}^{*}[R_{t}])=s^{t}_{t+1*}(\alpha_{t+1}),$$ where $s^{t}_{t+1}$ is defined accordingly. It remains to compute $\iota_{t*}(\pi_{t}^{*}\alpha_{t}\cup\iota_{t}^{*}[R_{t}])$. By functoriality of the Gysin sequence for continuous \'etale cohomology with support, we get $\iota_{t}^{*}[R_{t}]=\iota_{t}^{*}\cl_{\mathcal{S}}^{1}([R_{t}])\in H^{2}(R_{t},\Z_{\ell_{i}}(1))$. Since $[V]+\sum_{j=1}^{N}[R_{j}]+[W]\sim 0$ on $\mathcal{S}$, we find $$\cl^{1}_{\mathcal{S}}([R_{t}])=-\cl^{1}_{\mathcal{S}}([V])-\sum_{j\neq t}\cl^{1}_{\mathcal{S}}([R_{j}])-\cl^{1}_{\mathcal{S}}([W]).$$ We thus deduce the equality $\iota_{t}^{*}([R_{t}])=-\cl^{1}_{R_{t}}([C_{t-1,t}])-\cl^{1}_{R_{t}}([C_{t,t+1}])$ for $1\leq t\leq N$. By making the above substitution, we get $$\iota_{t*}(\pi_{t}^{*}\alpha_{t}\cup\iota_{t}^{*}[R_{t}])=-\iota_{t*}(\pi_{1}^{*}\alpha_{1}\cup\cl^{1}_{R_{t}}([C_{t-1,t}]))-\iota_{t*}(\pi_{t}^{*}\alpha_{t}\cup\cl^{1}_{R_{1}}([C_{t,t+1}]))=-s^{t}_{t-1*}(\alpha_{t})-s^{t}_{t+1*}(\alpha_{t}),$$ where the last equality follows again by the projection formula. Finally, $\tilde{\Phi}(\alpha_{1},\alpha_{2},\ldots,\alpha_{N})$ is given by \begin{align*}(s^{1}_{2*}(\alpha_{2})-s^{1}_{0*}(\alpha_{1})-s^{1}_{2*}(\alpha_{1}),\ldots,s^{t}_{t-1*}(\alpha_{t-1})+s^{t}_{t+1*}(\alpha_{t+1})-s^{t}_{t-1*}(\alpha_{t})-s^{t}_{t+1*}(\alpha_{t}),\ldots)\end{align*} and pushing forward with respect to $\{\pi_{j}\}$ yields the desired element. The proof is complete.\end{proof}
    
    \begin{lemma}\label{property-2} Write $N=\tilde{d}\ell^{r}_{i}-1$ for some positive integers $r$ and $\tilde{d}$ with $\gcd (\tilde{d},\ell_{i})=1$. For $A=\Z_{\ell_{i}}$ the map $\Phi$ is injective. If $A=\Z/\ell_{i}^{r}$, then $\ker(\Phi)\cong J(C)[\ell_{i}^{r}]$.    
    \end{lemma}
    
    \begin{proof} By Lemma \ref{properties: Pi}, we find that $(\alpha_{1},\alpha_{2},\ldots,\alpha_{N})\in\ker(\Phi)$ if and only if $\alpha_{j}=j\cdot\alpha_{1}$ for all $1\leq j\leq N$ and $\alpha_{N-1}-2\alpha_{N}=(N-1)\alpha_{1}-2N\alpha_{1}=-(N+1)\alpha_{1}=-\tilde{d}\ell^{r}_{i}\alpha_{1}=0$. Since $H^{1}(C,\Z_{\ell_{i}}(0))$ is torsion free, we conclude that $\Phi$ is injective in this case. For $A=\Z/\ell_{i}^{r}$, the map $$\alpha\mapsto (\alpha,2\alpha,\ldots,j\alpha,\ldots,N\alpha)$$ yields an isomorphism $J(C)[\ell_{i}^{r}]\cong\ker(\Phi)$, where we use that $J(C)[\ell_{i}^{r}]\cong H^{1}(C,\Z/\ell_{i}^{r})$.\end{proof}
    
    \begin{lemma}\label{Property-3} In the setting of Lemma \ref{property-2}, we have an isomorphism $\coker(\Phi)\cong J(C)[\ell_{i}^{r}]$ for $A\in\{\Z_{\ell_{i}},\Z/\ell^{r}_{i}\}$.\end{lemma}\begin{proof} We begin by showing that $\coker(\Phi)$ is annihilated by $\ell_{i}^{r}$. Pick an element $\alpha\in H^{1}(C,\Z_{\ell_{i}}(0))$ and consider the tuple $\alpha_{j}\in H^{1}(C,\Z_{\ell_{i}}(0))^{\oplus N}$ whose $j$-entry is the element $\alpha$ and all the others are zero. A simple induction argument then yields $$\Phi(\sum_{j=1}^{m}j\cdot\alpha_{j})=(0,0,\ldots,0,-(m+1)\alpha,m\alpha,0,\ldots,0),$$ where the non-trivial entries of this tuple are $m$ and $m+1$ and where $1\leq m\leq N-1$. From this follows $\Phi(\sum_{j=1}^{N}j\cdot\alpha_{j})=(0,0,\ldots,0,-N\alpha,(N-1)\alpha)+\Phi(N\cdot\alpha_{N})=(0,\ldots,0,-(N+1)\alpha)=(0,\ldots,0,-\tilde{d}\ell_{i}^{r}\alpha)$. Since the integer $\tilde{d}$ is invertible in $\Z_{\ell_{i}}$, we find $$\Phi(-{(\tilde{d})}^{-1}\sum_{j=1}^{N}j\cdot\alpha_{j})=\ell_{i}^{r}\cdot\alpha_{N}.$$ Let $\beta:=-{(\tilde{d})}^{-1}\sum_{j=1}^{N}j\cdot\alpha_{j}$ and consider the system of equations \begin{align*}
        \Phi(\ell_{i}^{r}\alpha_{N})=\ell_{i}^{r}\cdot\alpha_{N-1}-2\Phi(\beta),\\ \Phi(\ell_{i}^{r}\alpha_{j})=\ell_{i}^{r}\cdot\alpha_{j-1}-2\ell_{i}^{r}\cdot\alpha_{j}+\ell_{i}^{r}\cdot\alpha_{j+1}\ \text{for}\ 2\leq j\leq N-1 \\
        \Phi(\ell_{i}^{r}\cdot\alpha_{1})=-2\ell_{i}^{r}\cdot\alpha_{1}+\ell_{i}^{r}\cdot\alpha_{2}.\end{align*} A recursive induction on the above system yields $\ell^{r}_{i}\cdot\alpha_{j}\in\im(\Phi)$ for all $1\leq j\leq N$. In particular, $\coker(\Phi)$ is $\ell_{i}^{r}$-torsion.\par Next we consider the commutative diagram with exact rows\begin{center}
            \begin{tikzcd}
0 \arrow[r] & {H^{1}(C,\Z_{\ell_{i}}(0))^{\oplus N}} \arrow[r, "\times\ell_{i}^{r}"] \arrow[d, "\Phi"] & {H^{1}(C,\Z_{\ell_{i}}(0))^{\oplus N}} \arrow[r] \arrow[d, "\Phi"] & {H^{1}(C,\Z/\ell_{i}^{r}(0))^{\oplus N}} \arrow[r] \arrow[d, "\overline{\Phi}"] & 0 \\
0 \arrow[r] & {H^{1}(C,\Z_{\ell_{i}}(0))^{\oplus N}} \arrow[r, "\times\ell_{i}^{r}"]                  & {H^{1}(C,\Z_{\ell_{i}}(0))^{\oplus N}} \arrow[r]                  & {H^{1}(C,\Z/\ell_{i}^{r}(0))^{\oplus N}} \arrow[r]                             & 0,
\end{tikzcd}
        \end{center}where $\overline{\Phi}$ stands for $\Phi$ with $A=\Z/\ell_{i}^{r}$. Since $\ker(\Phi)=0$, serpents lemma gives the exact sequence $$0\to\ker(\overline{\Phi})\to\coker(\Phi)\overset{\times\ell_{i}^{r}}{\to}\coker(\Phi)\to\coker(\overline{\Phi})\to0,$$ from which immediately one deduces the isomorphisms $$\ker(\overline{\Phi})\cong\coker(\Phi)\ \text{and}\ \coker(\Phi)\cong\coker(\overline{\Phi}).$$ To conclude, recall by Lemma \ref{property-2} that $\ker(\overline{\Phi})\cong J(C)[\ell_{i}^{r}]$. This finishes the proof.
        \end{proof}

 Finally, we are in the position to conclude the proof of Theorem \ref{thm:extending-brauer-classes-2}. Recall that we have a class $\bar{\alpha}_{i}\in H^{3}_{cont}(\mathcal{S},\Z_{\ell_{i}}(1))$ whose image in $H^{3}_{cont}(\mathcal{S}\setminus R,\Z_{\ell_{i}}(1))$ is killed by $\ell_{i}^{r_{i}}$ and further restricts to a generator $0\neq\delta(\alpha_{i})\in H^{3}(S_{\bar{\eta}},\Z_{\ell_{i}}(1))\cong\Z/\ell_{i}^{r_{i}}$. By Claim \ref{claim-3}, there is a canonical extension \begin{align}
            \label{extension}0\to\Z/\ell_{i}^{r_{i}}\to H^{3}_{cont}(\mathcal{S},\Z_{\ell_{i}}(1))\to\bigoplus_{j=1}^{N}H^{3}(R_{j},\Z_{\ell_{i}}(1))\to 0.\end{align} The push-forward map $\iota_{*}\colon H^{1}_{BM}(R,\Z_{\ell_{i}}(0))\to H^{3}_{cont}(\mathcal{S},\Z_{\ell_{i}}(1))$ can be identified with $$\bigoplus_{j=1}^{N}H^{1}(R_{j},\Z_{\ell_{i}}(0))\overset{\{\iota_{j*}\}}{\to}H^{3}(\mathcal{S},\Z_{\ell_{i}}(1))$$ (see \eqref{eq:isomorphism-R}), where the last map is injective by Lemma \ref{property-2}. Hence \eqref{extension} yields a short exact sequence \begin{align} \label{final-extension} 0\longrightarrow\Z/\ell_{i}^{r_{i}}\longrightarrow \frac{H^{3}_{cont}(\mathcal{S},\Z_{\ell_{i}}(1))}{\im(\iota_{*})}\overset{g:=\{\iota^{*}_{j}\}}{\longrightarrow} J(C)[\ell_{i}^{r}] \longrightarrow 0,
            \end{align} where we use $\coker(\{\iota^{*}_{j}\}\circ\{\iota_{j*}\})\cong\coker(\Phi)\cong J(C)[\ell_{i}^{r}]$ (see Lemma \ref{Property-3}) and where $r\geq 2r_{i}$ by Claim \ref{claim-1}.
            
            \par The element $[\bar{\alpha}_{i}]\in\coker(\iota_{*})\subset H^{3}(\mathcal{S}\setminus R,\Z_{\ell_{i}}(1))$ is $\ell_{i}^{r_{i}}$-torsion. Thus, we may write $g([\bar{\alpha}_{i}])=\ell_{i}^{r_{i}}c$ for some $c\in J(C)[\ell_{i}^{2r_{i}}]\subset J(C)[\ell_{i}^{r}]$. Pick a class $[\gamma]\in\coker(\iota_{*})$ with $g([\gamma])=c$. Then $g([\bar{\alpha}_{i}-\ell_{i}^{r_{i}}\gamma])=0$ and therefore $[\bar{\alpha}_{i}-\ell_{i}^{r_{i}}\gamma]=m[\tau]$ for some generator $\tau$ of $\Z/\ell_{i}^{r_{i}}\hookrightarrow H^{3}_{cont}(\mathcal{S},\Z_{\ell_{i}}(1))$ and integer $0\leq m<\ell_{i}^{r_{i}}$. It follows in $H^{3}_{cont}(\mathcal{S}\setminus R,\Z_{\ell_{i}}(1))$, we have the equality $\bar{\alpha}_{i}=\ell_{i}^{r_{i}}\gamma+m\tau$. Since $H^{3}(S_{\bar{\eta}},\Z_{\ell_{i}}(1))\cong\Z/\ell_{i}^{r_{i}}$, restricting further to $S_{\bar{\eta}}$ yields $\delta(\alpha_{i})=m\tau$. In particular, this implies that the restriction map $$\Z/\ell_{i}^{r_{i}}\cong H^{3}(\mathcal{S},\Z_{\ell_{i}}(1))_{\tors}\to H^{3}(S_{\bar{\eta}},\Z_{\ell_{i}}(1))$$ is isomorphism. Observe that any $\alpha\in H^{2}(\mathcal{S},\mu_{\ell^{r_{i}}_{i}})$, such that $\delta(\alpha)\in H^{3}(\mathcal{S},\Z_{\ell_{i}}(1))_{\tors}$ is a generator yields a Brauer class on $\mathcal{S}$ with the property, $\delta(\alpha|_{S_{\bar{\eta}}})$ generates $H^{3}(S_{\bar{\eta}},\Z_{\ell_{i}}(1))$. This proves \eqref{itm:property-ext}.
            
            \par It remains to show \eqref{itm:property-rest-zero}. We need only choose our lift $\alpha_{i}\in\Br(\mathcal{S})[\ell^{r_{i}}_{i}]$ from \eqref{itm:property-ext} in such a way that the restriction to $W$ vanishes, i.e. $\alpha_{i}|_{W}=0$, since the other components of $S_{0}$ are ruled. To this end, we consider the following commutative diagram with exact rows \begin{center}
                \begin{tikzcd}
0 \arrow[r] & {H^{2}(\mathcal{S},\Z_{\ell_{i}}(1))\otimes\Z/\ell_{i}^{r_{i}}} \arrow[r] \arrow[d, "\rho"] & {H^{2}(\mathcal{S},\mu_{\ell_{i}^{r_{i}}})} \arrow[r, "\delta"] \arrow[d, "\bar{\rho}"] & {H^{3}(\mathcal{S},\Z_{\ell_{i}}(1))_{\tors}} \arrow[r] \arrow[d, "=0"] & 0 \\
0 \arrow[r] & {H^{2}(W,\Z_{\ell_{i}}(1))\otimes\Z/\ell_{i}^{r_{i}}} \arrow[r]                             & {H^{2}(W,\mu_{\ell_{i}^{r_{i}}})} \arrow[r, "\delta"]                                   & {H^{3}(W,\Z_{\ell_{i}}(1))_{\tors}=0} \arrow[r]                         & 0.
\end{tikzcd}
            \end{center}
    We claim that both restriction maps $\rho$ and $\bar{\rho}$ are surjective. Clearly, it suffices to prove surjectivity only for $\rho$. The proper base change theorem (see \cite[Corollary 3.4]{milne}) implies $H^{2}(\mathcal{S},\Z_{\ell_{i}}(1))\cong H^{2}(S_{0},\Z_{\ell_{i}}(1))$. The Mayer--Vietoris exact sequence then yields \begin{align}\label{MV-2}
        H^{2}(S_{0},\Z_{\ell_{i}}(1))\to H^{2}(W,\Z_{\ell_{i}}(1))\oplus(\bigoplus_{j=1}^{N}H^{2}(R_{j},\Z_{\ell_{i}}(1))\oplus H^{2}(V,\Z_{\ell_{i}}(1))\to \bigoplus_{j=0}^{N}H^{2}(C_{j,j+1},\Z_{\ell_{i}}(1)). \end{align} 
        Recall that $W$ admits an $n$-fold cyclic branched covering along $C:=C_{0,1}\subset W$ and so $C\sim_{\rat} n\cdot D$ for some divisor $D$ on $W$. Pick a class $\beta\in H^{2}(W,\Z_{\ell_{i}}(1))$ and consider the intersection number $d:=(D\cdot\beta)\in\Z$. Let $f_{j}\in H^{2}(R_{j},\Z_{\ell_{i}}(1))$ be the class of a fibre of the projection $\pi_{j}\colon R_{j}\to C$ and let $H\in H^{2}(V,\Z_{\ell_{i}}(1))$ be the class of a hyperplane. It is easy to check that the tuple $(\beta,(nd)f_{1},\ldots,(nd)f_{N},dH)$ maps to zero via the last map of \eqref{MV-2}. Exactness of \eqref{MV-2} thus yields an element $\tilde{\beta}\in H^{2}(S_{0},\Z_{\ell_{i}}(1))$, such that $\tilde{\beta}|_{W}=\beta$. In particular, we find that $\rho$ is surjective.
        
        \par Assume that a lift $\alpha_{i}\in H^{2}(\mathcal{S},\mu_{\ell_{i}^{r_{i}}})$ such that $\delta(\alpha_{i})\in H^{3}(\mathcal{S},\Z_{\ell_{i}}(1))_{\tors}$ is a generator, satisfies $0\neq\alpha_{i}|_{W}\in H^{2}(W,\mu_{\ell_{i}^{r_{i}}})$. Since $\rho$ is surjective and $H^{2}(W,\Z_{\ell_{i}}(1))\otimes\Z/\ell_{i}^{r_{i}}\cong H^{2}(W,\mu_{\ell_{i}^{r_{i}}})$, we can choose $\tilde{\alpha}_{i}\in H^{2}(\mathcal{S},\Z_{\ell_{i}}(1))\otimes\Z/\ell_{i}^{r_{i}}$, such that $\tilde{\alpha}_{i}|_{W}=\alpha_{i}|_{W}$. Moreover, we have $\delta(\alpha_{i}-\tilde{\alpha}_{i})=\delta(\alpha_{i})$ and thus, replacing the Brauer class of $\alpha_{i}$ by the Brauer class of $\alpha_{i}-\tilde{\alpha}_{i}$, we obtain the desired lift. The proof of Theorem \ref{thm:extending-brauer-classes-2} is finally complete.\end{proof}
        
\begin{proof}[Proof of Theorem \ref{thm:extend-Brauer-class}] Let $n=\ell_{1}^{r_{1}}\ell^{r_{2}}_{2}\dots\ell^{r_{s}}_{s}$ be the prime factorization. By Theorem \ref{thm:extending-brauer-classes-2}, there is a semi-stable degeneration $\mathcal{S}\to\Spec\kappa[[t]]$ having the properties \eqref{itm:generic-fibre}, \eqref{itm:special-fibre} and in addition for all $1\leq i\leq s$, there is a class $\alpha_{i}\in\Br(\mathcal{S})[\ell^{r_{i}}_{i}]$, such that $\delta(\alpha_{i}|_{S_{\bar{\eta}}})\in H^{3}(S_{\bar{\eta}},\Z_{\ell_{i}}(1))_{\tors}\cong\Z/\ell^{r_{i}}_{i}$ is a generator and such that for every component $S_{0i}$ of the special fibre, $\alpha_{i}|_{S_{0i}}=0\in\Br(S_{0i})$. It follows that the class $\alpha:=\sum_{i=1}^{s}\alpha_{i}\in\Br(\mathcal{S})[n]$ satisfies \eqref{it:extend-brauer-class}. The proof is complete.\end{proof}

\begin{corollary}\label{cor:injectivity-exterior-product} Let $\kappa$ be an algebraically closed field of characteristic zero and let $n\geq 2$ be an integer. Let $\mathcal{S}\to\Spec\kappa[[t]]$ denote the degeneration from Theorem \ref{thm:extend-Brauer-class} and fix an algebraic closure $k$ of the fraction field $\kappa((t))$. Then the following holds: For any smooth projective variety $Z$ over $\kappa$ the exterior product map \begin{align}\label{eq:exterior-product}
\CH^{j}(Z)/n\longrightarrow \CH^{j+1}(S_{\bar{\eta}}\times_{k}Z_{k})/n,\ [z]\mapsto [\mathcal{L}\times z]\end{align} is injective, where $0\neq\mathcal{L}\in\NS(S_{\bar{\eta}})_{\tors}=\Pic(S_{\bar{\eta}})_{\tors}\cong\Z/n$ is any generator.
\end{corollary}

\begin{proof} The properties $(\text{P}1)$ and $(\text{P}2)$ in Theorem \ref{thm:injectivity-exterior-product-a} correspond to \eqref{itm:extend-class} and \eqref{itm:rest-zero} in Theorem \ref{thm:extend-Brauer-class}, respectively. Thus, the result follows immediately from Theorem \ref{thm:injectivity-exterior-product-a}.\end{proof}

\section{Proof of the main results}
Theorem \ref{thm:injectivity-result} will be deduced from the following.

\begin{theorem}\label{thm:injectivity-result-2} Let $n\geq 2$ be an integer. Let $Z$ be a complex smooth projective variety. Then there is a smooth complete intersection $Y\subset\mathbb{P}^{5}_{\mathbb{C}}$ of three invariant hypersurfaces $H_{i}\in|\mathcal{O}_{\mathbb{P}^{5}}(n)|^{\varphi_{n}}$, such that the automorphism $\varphi_{n}$ (see \eqref{eq:varphi}) acts freely on $Y$ and such that the following holds: If we set $S:=Y/\varphi_{n},$ then the exterior product map \begin{align}\label{ext-prod-A^{i}}
    \CH^{j}(Z)/n\longrightarrow \CH^{j+1}(S\times Z)/n,\ [z]\mapsto [\mathcal{L}\times z]\end{align} is injective, where $0\neq\mathcal{L}\in \Pic(S)_{\tors}\cong\Z/n$ is any generator.\end{theorem}
    
\begin{proof} We argue as in the proof of \cite[Theorem 1.3]{Sch-griffiths}. We may find a countable algebraically closed field $\kappa\subset\C$, such that $Z=Z_{\kappa}\times_{\kappa}\C$ for some smooth projective variety $Z_{\kappa}$ over $\kappa$. By
the rigidity theorem of Lecomte \cite{lec}, the base change map
\begin{align*}
    \CH^{j}(Z_{\kappa})/n\longrightarrow \CH^{j}(Z)/n\end{align*}
is an isomorphism.
\par Fix an algebraic closure $k$ of the fraction field $\kappa((t))$ and let $\mathcal{S}\to \Spec
\kappa[[t]]$
be the degeneration of Theorem \ref{thm:extend-Brauer-class}. By Corollary \ref{cor:injectivity-exterior-product} the exterior product map \eqref{eq:exterior-product} $$\CH^{j}(Z_{\kappa})/n\longrightarrow \CH^{j+1}(S_{\bar{\eta}}\times_{k}Z_{k})/n$$ is then injective. As the field $\kappa$ is countable, the extension $\kappa\subset\C$ factors
through $k\subset\C$. As before by \cite{lec}, we have that the base change map $$\CH^{j+1}(S_{\bar{\eta}}\times_{k}Z_{k})/n\cong \CH^{j+1}(S\times_{\C}Z)/n$$ is an isomorphism,
where $S=S_{\bar{\eta}}\times_{k}\C$. Combining the above, we finally see that the exterior product map \eqref{ext-prod-A^{i}}
$$\CH^{j}(Z)/n\cong \CH^{j}(Z_{\kappa})/n\longrightarrow \CH^{j+1}(S\times Z)/n$$
is injective. The proof is complete.\end{proof}

Let $X$ be a complex smooth projective variety. Recall that we have a transcendental Abel--Jacobi map (see \cite[Theorem 1.6]{Sch-refined}) on torsion cycles 
\begin{align}\label{Abel-Jacobi-map} \lambda^{i}_{tr}\colon \Griff^{i}(X)_{\tors}\longrightarrow H^{2i-1}(X,\Q/\Z(i))/N^{i-1}H^{2i-1}(X,\Q(i)),\end{align} where $N^{*}$ denotes Grothendieck's coniveau filtration. By \cite[Proposition 8.5]{Sch-refined}, it coincides with Griffiths transcendental Abel--Jacobi map (see \cite{griffiths}) on torsion cycles. We let $$\mathcal{T}^{i}(X):=\ker(\lambda^{i}_{tr}).$$

\begin{corollary}\label{cor:trivial-abel-jacobi-invariant}In the setting of Theorem \ref{thm:injectivity-result-2}, there is a canonical injection $$E_{n}^{j}(Z)\ \hooklongrightarrow\mathcal{T}^{j+1}(S\times Z)[n].$$\end{corollary}

\begin{proof} By Theorem \ref{thm:injectivity-result-2}, we have an injection 
\begin{align}\label{c-i}
            E_{n}^{j}(Z)\ \hooklongrightarrow A^{j+1}(S\times Z)[n],\ [z]\mapsto[\mathcal{L}\times z],\end{align} where we recall that $A^{j}(X)=\CH^{j}(X)/\sim_{\alg}$ is the Chow group modulo algebraic equivalence. We claim that this map factors through $\mathcal{T}^{j+1}(S\times Z)[n]$. Let $p\colon S\times Z\to S$ and $q\colon S\times Z \to Z$ be the two projections. Note that for any $[z]\in E_{n}^{j}(Z)$, the class $$\cl^{j+1}_{S\times Z}([\mathcal{L}\times z])=p^{*}\cl^{1}_{S}([\mathcal{L}])\cup q^{*}\cl^{j}_{Z}([z])$$ is zero, since $\mathcal{L}$ is annihilated by $n$, whereas $\cl^{j}_{Z}([z])$ is divisible by $n$. It follows that the image of \eqref{c-i} is contained in $\Griff^{j+1}(S\times Z)[n]$. The formula (cf. \cite[Lemma 3.6]{th-sch})$$\lambda^{j+1}_{tr}([\mathcal{L}\times z])=p^{*}\lambda^{1}([\mathcal{L}])\cup q^{*}\overline{\cl^{j}_{Z}([z])},$$ where $\lambda^{1}$ is the isomorphism $\CH^1(S)_{\tors} \cong H^{1}(S_{\et},\Q/\Z)$ induced from the Kummer sequence
(see \cite[$\S$3.2.1]{CS-brauer}) and where $\overline{\cl^{j}_{Z}([z])}\in H^{2j}(Z,\Q/\Z(j))$ denotes the image of the reduction modulo $n$ of $\cl^{j}_{Z}([z])$ via the natural map $H^{2j}(Z,\Z/n(j))\to H^{2j}(Z,\Q/\Z(j))$ implies that $[\mathcal{L}\times z]\in\mathcal{T}^{j+1}(S\times Z)[n]$ for all $[z]\in E_{n}^{j}(Z)$. This concludes the proof.\end{proof}

\begin{proof}[Proof of Theorem \ref{thm:injectivity-result}] By Theorem \ref{thm:injectivity-result-2}, and the proof of Corollary \ref{cor:trivial-abel-jacobi-invariant}, we have a canonical injection $$E_{n}^{j}(Z)\ \hooklongrightarrow\frac{\Griff^{j+1}(S\times Z)[n]}{n\Griff^{j+1}(S\times Z)[n^{2}]},\ [z]\mapsto[\mathcal{L}\times z].$$ Moreover, note that $\Griff^{j}(Z)/n\subset E_{n}^{j}(Z)$ (see \cite[Corollary 7.12 and Lemma 7.13]{Sch-refined}) and so the injectivity of \eqref{ext-prod-griffiths} follows.\end{proof}

        \begin{proof}[Proof of Theorem \ref{thm:ell-torsion-griffiths}] Let $C\subset\mathbb{P}^{2}_{\C}$ be a smooth quartic curve and write $JC$ for its Jacobian variety. By Theorem \ref{thm:injectivity-result}, we have a canonical injection $$\Griff^{2}(JC)/n\ \hooklongrightarrow \frac{\Griff^{3}(S\times JC)[n]}{n\Griff^{3}(S\times JC)[n^{2}]},\ [z]\mapsto[\mathcal{L}\times z],$$ where $0\neq\mathcal{L}\in\Pic(S)_{\tors}\cong\Z/n$ is any generator. The result (where we take $X:=S\times JC$) thus follows from Theorem \ref{thm:infinite-modulo-n}, since we can choose $C$, such that $\Griff^{2}(JC)/n$ contains infinitely many elements of order $n$.\end{proof}
        
        \begin{remark} Note that in Theorem \ref{thm:ell-torsion-griffiths}, the infinitely many torsion elements in $\Griff^{3}(X)$ have all trivial Abel-Jacobi invariant, by Corollary \ref{cor:trivial-abel-jacobi-invariant}.\end{remark}
        
        \begin{proof}[Proof of Corollary \ref{cor:ell-torsion-griffiths}] Fix an integer $m\geq 5$. Then the result is an immediate consequence of the projective bundle formula (see \cite[Theorem 3.3]{fulton}) applied to the smooth complex projective variety $X:=S\times JC\times\mathbb{P}^{m-5}_{\C}$, where $C$ and $S$ are as in the proof of Theorem \ref{thm:ell-torsion-griffiths}.\end{proof}
        
    

 \section*{Acknowledgements}
 The author thanks his doctoral supervisor Stefan Schreieder for introducing him to this problem and for answering his questions, as well as for helpful comments during the preparation of the paper. This project has received funding from the European Research Council (ERC) under the European Union's Horizon 2020 research and innovation program under grant agreement No 948066 (ERC-StG RationAlgic).
 

\end{document}